\documentclass[10pt]{amsart}
\usepackage[T1]{fontenc}
\usepackage{geometry}
\usepackage[latin1] {inputenc}
\usepackage{amsmath}
\usepackage{amsfonts, amssymb, textcomp}
\usepackage[colorlinks=false, linkcolor=red,urlcolor=green, citecolor=blue]{hyperref}
\usepackage{subeqnarray}
\usepackage{color}
\usepackage{enumerate}
\usepackage{latexsym}
\usepackage{fancyhdr}
\usepackage{longtable}
\usepackage{amsmath, amssymb}
\usepackage{graphicx}

\theoremstyle{definition}
\newtheorem{theorem}{Theorem}[section]
\newtheorem{theo}[theorem]{Theorem}
\newtheorem{lemma}[theorem]{Lemma}
\newtheorem{proposition}[theorem]{Proposition}
\newtheorem{corollary}[theorem]{Corollary}
\newtheorem{definition}[theorem]{Definition}

\newtheorem{remark}[theorem]{Remark}

\def\oo{\infty}
\def\to{\rightarrow}
\def\a{\alpha}
\def\b{\beta}
\def\de{\delta}

\def\ep{\varepsilon}
\def\o{\omega}

\def\ga{\gamma}

\def\ro{\rho}

\def\ba{\boldsymbol{a}}
\def\bb{\boldsymbol{b}}

\def\bg{\boldsymbol{g}}

\def\bh{\boldsymbol{h}}

\def\bl{\boldsymbol{l}}

\def\bs{\boldsymbol{s}}
\def\bt{\boldsymbol{t}}

\def\en{\subseteq}
\def\N{\mathbb{N}}

\def\R{\mathbb{R}}
\def\C{\mathbb{C}}      

\def\M{\mathbb{M}}
\def\H{\mathbb{H}} 
\def\L{\mathbb{L}} 

\def\G{\mathbb{G}}

\def\m{\textbf{m}} 

\def\h{\textbf{h}}
\def\bm{\textbf{m}}
\def\l{\ensuremath{\boldsymbol\ell}}

\def\para{\par\smallskip}
\def\parb{\par\medskip}

\newcommand{\RR}{\mathbb{R}}

\newcommand{\ZZ}{\mathbb{Z}}

\newcommand{\up}{\operatorname{up}}
\newcommand{\low}{\operatorname{low}}
\let\on=\operatorname

\title[Indices of O-regular variation for weight functions and weight sequences]
{Indices of O-regular variation for weight functions and weight
sequences}

\author[J.~Jim\'{e}nez-Garrido, J.~Sanz, and G.~Schindl]{Javier Jim\'{e}nez-Garrido, Javier Sanz and Gerhard Schindl}

\begin{document}
\begin{abstract}
A plethora of spaces in Functional Analysis (Braun-Meise-Taylor and Carleman ultradifferentiable and ultraholomorphic classes; Orlicz, Besov, Lipschitz, Lebesque spaces, to cite the main ones) are defined by means of a weighted structure, obtained from a weight function or sequence subject to standard conditions entailing desirable properties (algebraic closure, stability under operators, interpolation, etc.) for the corresponding spaces. The aim of this paper is to stress or reveal the true nature of these diverse conditions imposed on weights, appearing in a scattered and disconnected way in the literature: they turn out to fall into the framework of O-regular variation, and many of them are equivalent formulations of one and the same feature. Moreover, we study several indices of regularity/growth for both functions and sequences, which allow for the rephrasing of qualitative properties in terms of quantitative statements.
\end{abstract}

\date{\today}
\maketitle

\section{Introduction}

The motivation of this paper arises from the study of ultraholomorphic and ultradifferentiable classes of functions, which consist of
smooth or analytic functions defined in an appropriate region $G$ of $\R$, $\C$ or the Riemann surface of the logarithm $\mathcal{R}$ whose derivatives satisfy one of the following estimates for some or all $A>0$: 
\begin{equation}\label{eq:estimation.f.factorialIN}
 \sup_{z\in G, \,\, p\in\N_0}\frac{|f^{(p)}(z)|}{A^p M_p}<\oo, \qquad \text{or}  \qquad \sup_{z\in G, \,\, p\in\N_0} |f^{(p)}(z)| \exp\left(-\frac{\varphi^{*}_{\o}(Ap)}{A} \right) <\oo,
\end{equation}
where $\N_0=\{0,1,2, \dots\}=\N\cup\{0\}$, $\M=(M_p)_{p\in\N_0}$ is a sequence of positive real numbers, $\varphi^{*}_{\o}(x):=\sup\{x y-\o(e^y): y\ge 0\}$ and $\o:[0,\oo)\to[0,\oo)$, see~\cite{BonetMeiseMelikhov07, Komatsu73,MeiseTaylor88, Petzsche88,PetzscheVogt,RainerSchindlcomposition,RainerSchindlExtension17,Schindl16,SchmetsValdivia00}.
For instance, if $G=[a,b]$ is some compact interval of the real line and $\o(t)=t$ or $M_p=p!$, then the class of smooth functions such that \eqref{eq:estimation.f.factorialIN} holds for this choice coincides with the class of analytic functions in $[a,b]$.  \para

One of the main topics regarding these classes is the characterization of the features of the ultradifferentiable or ultraholomorphic class in terms of properties of the corresponding sequence $\M$ or function $\o$. 
It is worthy to mention that some authors, specially when dealing with asymptotic expansions, have defined these classes using another estimate, see~\cite{ChaumatChollet94,JimenezSanz,JimenezSanzSchindlLCSeqandPO,JimenezSanzSchindlInjectSurject,Sanzsummability,Sanzflat14,Thilliez03}.
They assume that for some or all $A>0$: 
\begin{equation}\label{eq:estimation.f.factorialOUT}
 \sup_{z\in G, \,\, p\in\N_0}\frac{|f^{(p)}(z)|}{A^p p! M_p}<\oo.
\end{equation}
In some sense it can be said that the sequence $\M$ measures the lack of analyticity in this situation. 
There is a link between the properties usually assumed for $\M$ and the ones for $\widehat{\M}= (p!M_p)_{p\in\N_0}$. However these relations are not always straight and some considerations need to be  made. Since our purpose is that this work can be applied to both situations, the results are stated in a general framework and suitable comments are provided showing how to use them in each case.\para

In this context, diverse conditions satisfied by $\M$ or by $\o$ have been independently introduced 
adapted to the problem tackled in each work. Frequently, there are not further considerations about the nature of the  different properties  and their connections are not perfectly understood.
In some recent works two indices associated to sequences have appeared: $\ga(\M)$, considered by V.~ Thilliez~\cite{Thilliez03}, and $\o(\M)$, defined by the second author in~\cite{Sanzflat14}.  
They have been shown to measure the limiting opening of the sectors in $\mathcal{R}$ such that for every sector of opening strictly smaller or, respectively bigger, the Borel map in the corresponding ultraholomorphic class  is surjective or, respectively injective, see~\cite{JimenezSanzSchindlInjectSurject}. Similarly, the authors 
have introduced an analogous index $\ga(\o)$ for the function $\o$ solving an extension problem for the pertinent classes,  see~\cite{JimenezSanzSchindlExtensionReal,JimenezSanzSchindlLaplace}.\para

This paper aims to untangle the true essence of these characteristics which have come out in the literature. The solution lies on the classical theory of regular variation, concretely on the notion of O-regular variation, see~\cite{BingGoldTeug89}.
With this tool, the equivalence between the distinct conditions is provided and these qualitative  properties are expressed in terms of some quantitative  values, the growth orders and  the Matuszewska indices, which turn out to coincide with the indices for functions and sequences mentioned above. Furthermore, the differences and similarities between the function approach and the sequence one are highlighted.\para

In the analysis of these conditions, we have found that, apart from the theory of ultraholomorphic and ultradifferentiable classes, they have repeatedly and independently appeared in their different forms in several areas of Functional Analysis, specially dealing with weighted structures. 
For instance, the N-functions defining the Orlicz spaces are usually assumed to satisfy conditions $\Delta_2$ and $\nabla_2$ which can be identified with the properties here studied, see \cite{RaoRen91}. The index $\ga(\M)$ and some of the properties 
for sequences we will deal with have also been shown to be important for the Stieltjes moment problem in general Gelfand-Shilov spaces, see~\cite{LastraSanz09}.
In the previous cases the mass of the function is concentrated at $\oo$, but there are also weighted spaces 
like weighted Bergman, Besov, Lebesgue, or Lipschitz spaces considered by O.~Blasco and other authors, see~\cite{Blasco12} and the references therein, where the mass of the corresponding weight function is concentrated at $0$,
and  similar properties appear, for example \cite[(2.1) and (2.2)]{Blasco12}. 
The same happens for the weighted H\"{o}lder classes studied by E.~M.~Dynkin in~\cite{Dynkin80}
where the modulus of continuity is regular if it satisfies certain conditions related to the ones treated in this paper.  
In these situations the results of this work could be applied after a suitable modification, 
see Remark~\ref{rema:weightsAT0} for further details.\para

Most of those weighted spaces are defined from the classical ones replacing the function $t\mapsto t^\a$ for some $\a>0$ by a general function $t\mapsto\o(t)$. The extension of the classical results to the weighted context highly depends on a power-like behavior of $\o$ which leads to the theory of regular variation whose purpose is the systematic study of such type of behaviors. Hence the scope of this work might go beyond these examples and the results have been stated  
from a quite abstract point of view so they can be applied to diverse situations.\para

At this point we start describing the main achievements obtained in this paper and how they are organized. The second section 
starts by recalling the elementary facts about weight functions, regular and o-regular variation. 
The first important result, Lemma~\ref{lemma.alpha.gamma}, establishes the equality between the upper Matuszewska index $\a(\sigma)$  
and the inverse of the aforementioned index $\ga(\sigma)$ under the basic assumption of the section: $\sigma:[0,\oo)\to[0,\oo)$ is nondecreasing with $\lim_{t\rightarrow\infty}\sigma(t)=\infty$. The main results of the section, Theorems~\ref{th:main.th.alpha} and~\ref{th:main.th.beta},
provide a list of equivalent conditions for some of the basic properties assumed for weight functions, for instance
$$(\o_1) \quad \sigma(2t)=O(\sigma(t)), \,\,\, \text{as}\,\, t\rightarrow\infty\quad\text{and} \quad (\o_{\on{snq}}) \quad \exists C\geq 1:\,\,\, \forall y>0, \,\,\displaystyle\int_1^{\infty}\frac{\sigma(y t)}{t^2}dt\le C\sigma(y)+C, $$
in the first case, see also Corollaries~\ref{coro:omsnq} and~\ref{coro:om1}, and 
$$(\o_6)\quad \exists H\geq 1: \,\,\,  \forall t\ge0,\,\, 2\sigma(t)\leq\sigma(H t)+H,$$
in the second case, see also Corollary~\ref{coro:om1}. Moreover, these results also connect these properties to the Matuszewska indices thanks to the almost monotonicity notions.
Finally, the last subsection is devoted to the study of the relation between the index $\ga$ of a function and the ones of its upper and lower Legendre conjugates.\para 

In the third section, after summarizing the basic facts about weight sequences, we recall the main points of the theory of regularly and o-regularly varying sequences described in the works of R. Bojani\'{c} and E.~Seneta~\cite{BojanicSeneta} and D.~Djur\v{c}i\'{c} and V.~Bo\v{z}in~\cite{DjurcicBozin}, respectively.  The first task carried out in this section has been introducing the notion of growth order and Matuszewska indices, to the best of our knowledge new, and proving some elementary properties of those values, see from Proposition~\ref{prop.nice.def.Mat.ind.ord.seq} to Remark~\ref{remark:shift.seq.Mat.ind.order}. In Theorem~\ref{th:equality.ga.beta.mu.omega}, it is shown that $\ga(\M)$ equals the lower Matuszewska index of the sequence of quotients $\m:=(m_p=M_{p+1}/M_p)_{p\in\N_0}$ and $\o(\M)$ coincides with   
the lower order of $\m$. The section concludes with Theorems~\ref{th:main.beta.sequences} and~\ref{th:main.alpha.sequences} where the strong nonquasianalytic and the moderate growth conditions:
$$\text{(snq)} \quad \exists B\geq 0: \,\,\,  \forall p \in\N_0, \,\, \sum^\oo_{q= p}\frac{M_{q}}{(q+1)M_{q+1}}\le B\frac{M_{p}}{M_{p+1}},$$
$$\text{(mg)} \quad \exists A\geq 0: \,\,\,  \forall p, q\in\N_0, \,\, M_{p+q}\le A^{p+q}M_{p}M_{q}, $$
are characterized in terms of the Matuszewska indices of $\m$ and compared with other conditions appearing in the literature. In the proof we have made use of  
Theorems~\ref{th:main.th.alpha} and~\ref{th:main.th.beta} although it is possible to show them directly
as it has been partially done by the first author in~\cite{JimenezPhD} with similar arguments.\para

The fourth section aims to compare the weight function approach with the weight sequence one through the counting function of the sequence of quotients $\nu_{\m}(t):=\max\{j\in\N:m_{j-1}\le t \}$ and the associated function $\o_{\M}(t):=\sup_{p\in\N_0}\log\left( t^p/M_p\right)$ for all $t\geq 0$. In the first subsection, the duality relation between the orders and Matuszewska indices of these functions and the ones of the corresponding weight sequence is explained. In the second subsection, Corollary~\ref{coro:OmhatMOmMgamma} shows under suitable assumptions that $\ga(\o_{\widehat{\M}})= \ga(\o_{\M})+1$ by means of the Legendre conjugate. The strongly regular sequences, which appear in different issues, see the references in Subsection~\ref{subsect:SRS}, are characterized in  Corollary~\ref{Coro.SRS.Charact} in terms of the Matuszewska indices of $\m$, $\nu_\m$ and $\o_{\M}$. The section ends analyzing the connection between nonzero proximate orders and weight functions. Nonzero proximate orders have been used by A. Lastra, S. Malek and the second author in~\cite{Sanzsummability}  to develop a summability theory for ultraholomorphic classes defined in terms of a weight sequence. Thanks to Corollary~\ref{coro:weightfunction.proximateorders}, we see that the information for ultraholomorphic classes defined in terms of a weight function is the same as the one from the weight sequence case.\para

The last section contains a counter-example of a weight sequence $\M$ such that $\ga(\M)$ and $\ga(\o_{\M})$ do not coincide, so the corresponding properties associated with these indices are not equivalent. This fact clarifies the duality relations  described in the previous section.

\section{Weight functions and O-Regular variation}\label{WeightFandORV}

\subsection{Weight functions \texorpdfstring{$\omega$}{Om} in the sense of Braun-Meise-Taylor}\label{weightfunctionclasses}

A function $\omega:[0,\infty)\rightarrow[0,\infty)$ is called a {\itshape weight function} if it is continuous, nondecreasing, $\omega(0)=0$, and $\lim_{t\rightarrow\infty}\omega(t)=\infty$.
If in addition $\o(t)=0$ for all $t\in[0,1]$ we say that $\o$ is a \textit{normalized weight function}. Moreover, the following conditions are often considered when dealing with weighted spaces of functions:
\begin{itemize}
\item[\hypertarget{om1}{$(\omega_1)}$] $\omega(2t)=O(\omega(t))$ as $t\rightarrow\infty$.

\item[\hypertarget{om2}{$(\omega_2)$}] $\omega(t)=O(t)$ as $t\rightarrow\infty$.

\item[\hypertarget{om3}{$(\omega_3)$}] $\log(t)=o(\omega(t))$ as $t\rightarrow\infty$.

\item[\hypertarget{om4}{$(\omega_4)$}] $\varphi_{\omega}:t\mapsto\omega(e^t)$ is a convex function on $\RR$.

\item[\hypertarget{om5}{$(\omega_5)$}] $\omega(t)=o(t)$ as $t\rightarrow\infty$.

\item[\hypertarget{om6}{$(\omega_6)$}] there exists $H\geq 1$ such that for all $t\ge0$ $2\omega(t)\le\omega(H t)+H$.

\item[\hypertarget{om7}{$(\omega_7)$}]  there exists $H, C\geq 0$ such that for all $t\ge0$  $ \omega(t^2)\le C\omega(H t)+C.$

\item[\hypertarget{omnq}{$(\omega_{\text{nq}})$}] $\displaystyle \int_1^{\infty}\frac{\omega(t)}{t^2}dt<\infty.$

\item[\hypertarget{omsnq}{$(\omega_{\text{snq}})$}]  there exists $C\geq 1$ such that for all $y>0$, $\displaystyle\int_1^{\infty}\frac{\omega(y t)}{t^2}dt\le C\omega(y)+C$.
\end{itemize}

The classical examples are the well-known \textit{Gevrey weights of index $s>1$}, $\o(t)=t^{1/s}$, which define the Gevrey classes, they satisfy all listed properties except \hyperlink{om7}{$(\omega_7)$}. Another interesting example is $\o(t)=\max\{0,\log(t)^s\}$, $s>1$, which satisfies all listed properties except \hyperlink{om6}{$(\omega_6)$}. 
Note that, for a weight function, concavity implies subadditivity, i.e.
$\omega(s+t)\le \omega(s)+\omega(t)$, and this yields \hyperlink{om1}{$(\omega_1)$}.\para

Let $I$ be an unbounded subinterval of $[0,\infty)$ and $\sigma,\tau:I\to[0,\oo)$ be any pair of measurable functions, we  call them \textit{equivalent} and we write $\sigma\hypertarget{sim}{\sim}\tau$ if there exists $C\geq 1$ such that for every $t\in I$ 
$$C^{-1}\tau(t)-C\leq \sigma (t) \leq C\tau(t)+C.$$

\subsection{Regular variation and O-Regular variation}

The notion of regular variation was formally introduced in 1930 by J.~Karamata~\cite{Karamata30} and has several applications in analytic number theory, complex analysis and, specially, in probability. The proofs of most of the results in this subsection are gathered in the book of N.~H.~Bingham, C.~M.~Goldie and J.~L.~Teugels~\cite{BingGoldTeug89}.  Until the end of this subsection, we assume that\parb

\centerline{\textbf{$f:[A,\oo)\to(0,\oo)$, with  $A> 0$, is a measurable function.}} \parb
 
We say that $f$ is \textit{regularly varying} if for every $\lambda\in(0,\oo)$, 
 \begin{equation}\label{equation.def.RV}
 \lim_{x\to\oo} \frac{f(\lambda x)}{f(x)} = g(\lambda)\in(0,\oo).
 \end{equation}

 There are three main results of this theory: the Uniform Convergence, the Representation and the Characterization Theorems~\cite[Th. \ 1.3.1, Th. \ 1.4.1 and Th.\ 1.5.2]{BingGoldTeug89}, the last one states that, if $f$ is regularly varying, then there exists $\ro\in\R$ such that the function $g(\lambda)$ in~\eqref{equation.def.RV} is equal to $\lambda^\ro$. In this case, $\ro$ is called the \textit{index of regular variation of $f$}, we write $f\in R_\ro$ and $RV:=\cup_{\ro>0} R_\ro$. If $\ro=0$, then $f$ is said to be \textit{slowly varying}.\para

Consequently, the behavior of a regularly varying function at $\oo$ is in some sense similar to the behavior of a power-like function. If for $n\in\N$, $\log_n x$ denotes the $n-$th iteration of the logarithm, given $n_i\in\N$  and $\a_i\in\R$ for $i=0,1\dots ,k$, the classical example of a regularly varying function is
$$f(x)=x^{\a_0} (\log_{n_1} x)^{\a_1} (\log_{n_2} x)^{\a_2} \cdots (\log_{n_k} x)^{\a_k}.  $$
 

For some of our purposes, the theory of regular variation is too restrictive and one may ask what remains valid if we replace
$\lim$ by $\limsup$ and $\liminf$ in~\eqref{equation.def.RV}. This extension of the class of regularly varying functions  was defined by J.~Karamata, V. G.~Avakumovi\'{c}, considered by W.~Matuszewska~\cite{Matuszewska} and W.~Feller~\cite{Feller} and systematically studied by S.~Aljan\v{c}i\'c and I.~D.~ Arandjelovi\'c~\cite{AljAra77} in 1977. 
We say that $f$ is \textit{O-regularly varying} \cite[p. 61]{BingGoldTeug89} if 
 \begin{equation}\label{equation.def.ORV}
 0< f_{\low}(\lambda):= \liminf_{x\to\oo}  \frac{f(\lambda x)}{f(x)} \leq f^{\up}(\lambda):= \limsup_{x\to\oo} \frac{f(\lambda x)}{f(x)} <\oo
 \end{equation}
for every $\lambda\geq 1$, and we write $f\in ORV$.  This weaker notion preserves several desirable properties, in particular, the three main theorems of regular variation have their adapted version.

\begin{remark}\label{rema.def.ORVF.limsup.lambda.pos}
 We observe that $f_{\low}(\lambda)=1/f^{\up}(1/\lambda)$ for every $\lambda\geq 1$. Consequently, if $f\in ORV$, then \eqref{equation.def.ORV} holds for every $\lambda\in (0,\oo)$ and we deduce that $RV\en ORV$. Moreover, $f\in ORV$ if and only if 
 $$ f^{\up}(\lambda)= \limsup_{x\to\oo} \frac{f(\lambda x)}{f(x)} <\oo\qquad \text{for every $\lambda\in (0,\oo)$.}$$
\end{remark}

In this general context,  the index $\ro$ of regular variation is split into two values, the  Matuszewska indices \cite[p. 68]{BingGoldTeug89}. For any positive function, the \textit{upper Matuszewska index $\a(f)$} is defined by
\begin{equation*}
\a(f):=\inf\left\{\a\in\R;\,\, \exists C_\a>0\,\, \text{s.t.}\,\, \forall\Lambda>1,\,\, \limsup_{x\to\oo} \sup_{\lambda\in[1,\Lambda]} \frac{f(\lambda x)}{\lambda^\a f(x)}\leq C_\a \,  \right\}
\end{equation*}
 and the \textit{lower Matuszewska index $\b(f)$} by
\begin{equation*}
\b(f):=\sup\left\{\b\in\R;\,\, \exists D_\b>0\,\, \text{s.t.}\,\, \forall\Lambda>1,\,\, \liminf_{x\to\oo} \inf_{\lambda\in[1,\Lambda]}  \frac{f(\lambda x)}{\lambda^\b f(x)}\geq D_\b,\,  \right\}.
\end{equation*}

\begin{remark}\label{rema:inf.sup.empty}
Since these sets are either empty or unbounded intervals from above for $\a$ and, respectively form below for $\b$, we are allowed to use the classical conventions $\inf \emptyset=\sup\R=\oo$ and $\inf\R=\sup\emptyset=-\oo$.
\end{remark}

Moreover, the inequality $\b(f)\leq \a(f)$ always holds.  The finiteness of these indices characterizes O-regular variation.

\begin{theorem}[\cite{BingGoldTeug89}, Th. 2.1.7]\label{th:Ch.ORVF}
 $f$ is \textit{O-regularly varying} if  and only if $\a(f)<\oo$ and $\b(f)>-\oo$.
\end{theorem}

These indices admit a nicer and useful representation in terms of some almost monotonicity properties. We call a function $h:[a,+\infty)\rightarrow[0,+\infty)$, with $a \geq 0$, {\itshape almost increasing,} if there exists some $M>0$ such that $h(x)\le Mh(y)$ for all $a\le x\le y<+\infty$. Analogously we call $h$ {\itshape almost decreasing,} if there exists some $m>0$ such that $m h(y)\le h(x)$ for all $a\le x\le y<+\infty$.

\begin{theorem}[\cite{BingGoldTeug89} Th. 2.2.2]\label{th:almostmonotonerep.alpha.beta} For $f$ as above, 
\begin{equation*}\label{matu1}
\alpha(f)=\inf\{\alpha\in\RR: x\mapsto\frac{f(x)}{x^{\alpha}}\;\text{is almost decreasing}\},
\end{equation*}
\begin{equation*}\label{matu2}
\beta(f)=\sup\{\beta\in\RR: x\mapsto\frac{f(x)}{x^{\beta}}\;\text{is almost increasing}\}.
\end{equation*}
\end{theorem}

\begin{remark}\label{rema:ordersROMU}
These indices are related to the classical {\itshape  lower order $\mu(f)$}  and {\itshape  upper order $\ro(f)$} defined by 
\begin{equation*}
\mu(f):=\liminf_{x\rightarrow\infty}\frac{\log(f(x))}{\log x}, \qquad \ro(f):=\limsup_{x\rightarrow\infty}\frac{\log f(x)}{\log x},
\end{equation*}
in the following way: 
$\b(f)\le \mu(f) \leq \ro (f) \leq \a(f),$
see \cite[Prop.  2.2.5]{BingGoldTeug89}.  If $f\in R_\ro$, then 
$\b(f)=\mu(f)=\ro(f)=\a(f)=\ro$. However, in general, the inequalities are strict and the equality of indices and orders does not guarantee regular variation.  
\end{remark}

As a consequence of the  almost monotone characterization the following properties are deduced.

\begin{remark}\label{rema.properties.Matuszewska.indices} If $f,g:[A,\oo)\to(0,\oo)$, with  $A > 0$, are measurable functions with
\begin{equation}\label{eq:equiv.functionsORV}
  0<\liminf_{x\to\oo} \frac{f(x)}{g(x)}\le \limsup_{x\to\oo} \frac{f(x)}{g(x)}<\oo,
\end{equation}
 then $\b(f)=\b(g)$, $\mu(f)=\mu(g)$, $\ro(f)=\ro(g)$ and $\a(f)=\a(g)$.\para
 
 For any $s>0$ and $r\in \R$, we put $f_s(t):=f(t^s)$,  $f^s(t):= (f(t))^s$ and $p_r(t):=t^r$, and we see that
 $$\b(f_s)=s\b(f),\qquad \b(f^s)=s\b(f),\qquad  \b(f\cdot p_r)= r+\b(f).$$
 The same is valid if we replace the index $\b$ by $\mu$, $\ro$ or $\a$. Moreover, if $s<0$, then 
 $\a(f^s)=s\b(f)$, $\ro(f^s)=s\mu(f)$, $\mu(f^s)=s\ro(f)$ and $\b(f^s)=s\a(f)$. 
\end{remark}

These indices might be difficult to handle. Fortunately,  if some additional property holds, then
a combination of \cite[Th. 2.1.5, Coro. 2.1.6, Th. 2.1.7]{BingGoldTeug89} leads to this suitable
 representation.

\begin{theorem}\label{th:convenient.rep.alpha.beta}
If $\a(f)<\oo$ or $\b(f)>-\oo$, then
$$\alpha(f)=\lim_{\lambda\to\oo} \frac{\log(f^{\on{up}}(\lambda))}{\log \lambda} =\inf_{\lambda>1}\frac{\log(f^{\on{up}}(\lambda))}{\log \lambda },$$
$$\beta(f)=\lim_{\lambda\to\oo}\frac{\log(f_{\on{low}}(\lambda))}{\log \lambda } =\sup_{\lambda>1}\frac{\log(f_{\on{low}}(\lambda))}{\log \lambda}.$$
\end{theorem}

Hence, this theorem holds for O-regularly varying functions but also under weaker conditions.   In particular,
the next lemma, which is a consequence of Theorem~\ref{th:almostmonotonerep.alpha.beta}, ensures that 
 Theorem~\ref{th:convenient.rep.alpha.beta} is available for monotone (not necessarily ORV) functions. This situation occurs
when dealing with weight functions or sequences.

\begin{lemma}\label{lemma:beta.alpha.monotone.function}
\begin{enumerate}[(i)]
 \item If $f$ is nondecreasing, then $\b(f)\ge 0$.  Hence $f\in ORV$ if and only if $\a(f)<\oo$.
 \item If $\b(f)>0$, then there exists an nondecreasing function $g$  satisfying \eqref{eq:equiv.functionsORV}. 
  \item If $f$ is nonincreasing, then $\a(f)\le 0$. Hence $f\in ORV$ if and only if $\b(f)>-\oo$. 
 \item If $\a(f)<0$, then there exists a nonincreasing function  $g$  satisfying \eqref{eq:equiv.functionsORV}.  
\end{enumerate}
\end{lemma}


\subsection{Index \texorpdfstring{$\gamma$}{Gamma} for nondecreasing functions and its relation to \texorpdfstring{$\a$}{Alpha}}\label{growthindexgamma}

In \cite{JimenezSanzSchindlExtensionReal,JimenezSanzSchindlLaplace}, the index $\ga(\o)$ was introduced in order to measure the limit opening which the Borel map defined in the corresponding ultraholomorphic class in a sector of the Riemann surface of the logarithm is surjective for. The definition is based on \cite[Prop.\ 1.3]{MeiseTaylor88} and has
the same spirit as V.~Thilliez's index $\ga(\M)$ for sequences considered in~\cite{Thilliez03}.\para

This index was originally defined for a weight function $\o$, but due to the general approach of this paper 
 we will work in a more general framework. Until the end of the section we deal with:\parb

\centerline{\textbf{$\sigma:[0,\oo)\to[0,\oo)$, nondecreasing with $\lim_{t\rightarrow\infty}\sigma(t)=\infty$.}}\parb

Although even weaker assumptions may be considered, this approach is enough to cover at the same time the weight function and the weight sequence case.
We can also treat other weighted structures introduced and used in different fields of Functional Analysis as it has been explained in the introduction, see also Remark~\ref{rema:weightsAT0}.
Let $\sigma$ and $\gamma>0$ be given, we say that $(P_{\sigma,\gamma}) $ holds if there exists $K>1$ such that $$\limsup_{t\to\oo} \frac{\sigma(K^{\gamma}t)}{\sigma(t)} <K.$$
{\itshape Note:} If $(P_{\sigma,\gamma})$ holds for some $K>1$, then also $(P_{\sigma,\gamma'})$ is satisfied for all $\gamma'\le\gamma$ with the same $K$ and since $\sigma$ is nondecreasing  we might restrict ourselves to $\gamma>0$. \parb

Finally, we put
\begin{equation}\label{newindex2}
\gamma(\sigma):=\sup\{\gamma>0: (P_{\sigma,\gamma})\;\;\text{is satisfied}\}.
\end{equation}
If none condition $(P_{\sigma,\gamma})$ holds true, then  we put $\gamma(\sigma):=0$.

\begin{remark}\label{rema:Matuszewska.indices.welldef}
 We want to compare $\gamma(\sigma)$ with the Matuszewska indices introduced above. In the classical literature, $\a(\sigma)$ and $\b(\sigma)$  are defined only for positive functions. For $\sigma$ as above, by $\a(\sigma)$ and $\b(\sigma)$ we mean the corresponding indices of the restriction of $\sigma$ to some interval $[A,\oo)$, with $A>0$, where $\sigma(t)>0$, which is possible because $\sigma$ is nondecreasing with $\lim_{t\rightarrow\infty}\sigma(t)=\infty$ and, by Remark~\ref{rema.properties.Matuszewska.indices}, the indices do not depend on the restriction considered.
\end{remark}

\begin{lemma}\label{lemma.alpha.gamma} The relation $$\displaystyle \a(\sigma)=\frac{1}{\ga(\sigma)}$$ 
holds. Since $\ga(\sigma),\a(\sigma)\in[0,\oo]$, for the extreme values this means
that $\ga(\sigma)=\oo$ if and only if $\a(\sigma)=0$ and $\ga(\sigma)=0$ if and only if $\a(\sigma)=\oo$.
\end{lemma}

\begin{proof}
By Lemma~\ref{lemma:beta.alpha.monotone.function}.(i), $\beta(\sigma)\ge 0$, so Theorem~\ref{th:convenient.rep.alpha.beta} can be applied and write $$\alpha(\sigma)=\inf_{\lambda>1} \frac{\log(\sigma^{\on{up}}(\lambda))}{ \log \lambda}.$$ We  rewrite the definition of the index $\gamma(\sigma)$ defined in \eqref{newindex2} to obtain:
$$\gamma(\sigma)=\sup\{\gamma>0;\;\exists\;\Lambda> 1:\;\limsup_{t\rightarrow\infty}\frac{\sigma(\Lambda t)}{\sigma(t)}<\Lambda^{1/\gamma}\}= (\inf\{\tau>0;\;\exists\;\Lambda> 1:\; \sigma^{\up} (\Lambda) <\Lambda^{\tau}\})^{-1}.$$
Let now $\alpha>\alpha(\sigma)$, then there exists some $\lambda>1$ such that $(\log(\sigma^{\on{up}}(\lambda))/\log \lambda)<\alpha$ and consequently $(\gamma(\sigma))^{-1}\le\alpha$.
Conversely, let $\tau>(\gamma(\sigma))^{-1}$ be given, then there exists some $\Lambda> 1$ such that $(\log(\sigma^{\on{up}}(\Lambda))/\log \Lambda))<\tau$ and so $\alpha(\sigma)<\tau$ follows.
\end{proof}

From this connection and according to Remark~\ref{rema.properties.Matuszewska.indices} we deduce the following properties, most of them have been obtained in~\cite{JimenezSanzSchindlLaplace} directly from the definition:

\begin{itemize}
\item[$(i)$] For any $s>0$, $\ga(\sigma^s)=\ga(\sigma)/s$ and $\ga(\sigma_s)=\ga(\sigma)/s$ where $\sigma_s(t)=\sigma(t^s)$ and  $\sigma^s(t)= (\sigma(t))^s$.

\item[$(ii)$] Let $\omega,\sigma:[0,\oo)\to[0,\oo)$, nondecreasing with $\lim_{t\rightarrow\infty}\o(t)=\lim_{t\rightarrow\infty}\sigma(t)=\infty$ and $\sigma\hyperlink{sim}{\sim}\omega$ be given. Then $\sigma$ and $\o$ satisfy \eqref{eq:equiv.functionsORV} so $\gamma(\sigma)=\gamma(\omega)$. 
\end{itemize}

\subsection{Main theorems} S.~Aljan\v{c}i\'c and I.~D.~Arandjelovi\'c~\cite{AljAra77} give several equivalent representations of the indices $\a(f)$ and $\b(f)$ for O-regularly varying functions. In Theorems~\ref{th:main.th.alpha} and~\ref{th:main.th.beta}, this information is extended for monotone, but not necessarily O-regularly varying, functions. In Corollaries~\ref{coro:omsnq}, ~\ref{coro:om1} and~\ref{coro:om6} we deduce the relation between these indices and some of the classical conditions usually assumed for weight functions. 
The proof of the main results is based on  N.~K.~Bari, S.~B.~Stechkin \cite[Lemmas 2 and 3]{BariSteckin} and  R.~Meise, B.~A.~Taylor \cite[Prop. 1.3]{MeiseTaylor88}.\para

We will denote by $\N$ the set $\{1,2,3, \dots\}$, write $\N_0=\N\cup\{0\}$ and  $\lceil x \rceil:=\min\{ k\in\ZZ; x\leq k \}$ for $x\in\R$.

\begin{theorem}\label{th:main.th.alpha}
Let $\sigma:[0,\oo)\to[0,\oo)$, nondecreasing with $\lim_{t\rightarrow\infty}\sigma(t)=\infty$ and $\a>0$ be given. 
We take $a\geq 0$ such that $\sigma(x)>0$ for $x\ge a$. Then the following are equivalent:
\begin{itemize}
\item[(i)] there exists $C>0$ such that  $\displaystyle \int_1^{\infty}\frac{\sigma(y t)}{t^{1+\a}}dt\le C\sigma(y)+C$ for all $y>0$,
\item[(ii)] there exists a nondecreasing function $\kappa: [0,+\infty)\longrightarrow[0,+\infty)$ such that $\sigma {\hyperlink{sim}\sim}\kappa$, $\displaystyle \kappa(0)=\frac{\sigma(0)}{\a^2}$, $\kappa$ satisfies (i) and $\kappa(t^{1/\a})$ is concave, 
\item[(iii)] $ \displaystyle \lim_{\varepsilon\rightarrow 0}\limsup_{t\rightarrow+\infty}\frac{\varepsilon^\a \sigma(t)}{\sigma(\varepsilon t)}=0,$ 
\item[(iv)] there exists  $K>1$ such that $ \displaystyle \limsup_{t\rightarrow+\infty}\frac{\sigma(Kt)}{\sigma(t)}<K^\a,$
\item[(v)] $\ga(\sigma)>1/\a$,
\item[(vi)] $\a(\sigma)<\a$ (with the convention in Remark~\ref{rema:Matuszewska.indices.welldef}),
\item[(vii)] there exists $\ga\in(0,\a)$ such that $\displaystyle t \mapsto\frac{\sigma(t)}{t^\ga}$ is almost decreasing in $[a,\oo)$ if $a>0$, and in $[\ep,\oo)$ for any $\ep>0$ if $a=0$,
\item[(viii)]  there exists $C>0$ such that  $\displaystyle  \int_a^{y}\frac{t^{\a}}{\sigma(t)}\frac{dt}{t}\le \frac{Cy^{\a}}{\sigma(y)}$ for all $y\geq a$,
\item[(ix)]  there exists $C>0$ such that $\displaystyle \sum_{k=\lceil a \rceil+1}^p \frac{k^{\a-1}}{\sigma(k)}\leq C \frac{p^\a}{\sigma(p)}$ for all $p\in\N$ with $p\geq \lceil a \rceil +1 $, 
\item[(x)] for every $\theta\in(0,1)$ there exists $k\in\N$, $k\geq 2$, such that 
$\sigma (kp) \leq \theta k^\a  \sigma \left( p \right)$ for every $p\in\N$ with $p \geq  \lceil a \rceil +1  $,
\item[(xi)] there exists $k\in\N$, $k\geq 2$  such that $\displaystyle \limsup_{p\in\N,\, p\rightarrow\infty}\frac{\sigma(kp)}{\sigma(p)}< k^\a$,
\item[(xii)] there exists $C>0$ such that  
$\displaystyle \sum_{k=p}^\oo \frac{\sigma(k)}{k^{1+\a}}\leq C \frac{\sigma(p)}{p^\a}$ for every $p\in\N$  with $p\geq \lceil a \rceil $.

\end{itemize}

\end{theorem}

\begin{proof}

(i) $\Rightarrow$ (ii)    For all $y>0$ we define
   \begin{equation*}
\kappa_0 (y) =\xi_{\sigma}(y):=\int_1^{\infty}\frac{\sigma(yt)}{t^{1+\a}}dt= y^\a \int_y^{\infty}\frac{\sigma(s)}{s^{1+\a}}ds.
\end{equation*}
Using that $\sigma$ is nondecreasing, we observe that $\kappa_{0}$ is continuous nondecreasing in  $[0,+\oo)$, $\kappa_{0}(0)=\sigma(0)/\a$ and $\kappa_{0} (y)\geq \sigma(y)/\a$ for all $y\geq 0$. By (i), we see that $\kappa_{0} (y)\leq C \sigma(y) + C$ for all $y>0$, so $\sigma{\hyperlink{sim}\sim}\kappa_{0}$, then one can check that $\kappa_{0}$ also satisfies (i) for some constant $C'$. \par

Iterating the procedure, we construct $\kappa (y):=\xi_{\kappa_0}(y)\in \mathcal{C}^{1} (0,\oo)$, $\kappa$ is nondecreasing, $\kappa(0)=\sigma(0)/\a^2$, $\kappa\sim \kappa_0\sim \sigma$, $\kappa$ satisfies (i). Finally, for all $y>0$ we see that
$$(\kappa(y^{1/\a}) )'=\left(  y \int_{y^{1/\a}}^{\infty}\frac{\kappa_0(s)}{s^{1+\a}}ds \right)'= \int_{y^{1/\a}}^{\infty}\frac{\kappa_0(s)}{s^{1+\a}}ds -\frac{\kappa_0(y^{1/\a})}{\a y}=  \int_{y^{1/\a}}^{\infty}\frac{\kappa_0(s)- \kappa_0(y^{1/\a})}{s^{1+\a}}ds.$$
Since $\kappa_0(y^{1/\a})$ is nondecreasing we conclude that $(\kappa(y^{1/\a}))'$ is nonincreasing, so $\kappa(y^{1/\a})$ is concave.\para

(ii) $\Rightarrow$ (iii)  We will show that $\kappa$ satisfies (iii), 
and we conclude using that $\kappa\sim\sigma$. Since $\kappa(t^{1/\a})$ is concave and $\kappa(0)\geq0$, $\kappa((2t)^{1/\a}) \leq \kappa((2t)^{1/\a})+\kappa(0)\leq 2 \kappa (t^{1/\a})$  for all $t\ge 0$, so we put $A:=2^{1/\a}>1$ and we see that $\kappa(Ay)\leq A^\a \kappa(y)$ for all $t\geq 0$, or more generally, $A^{-n \a } \kappa(A^{n} y)\leq A^{-j \a } \kappa(A^{j} y)$ for all $y\geq 0$ and every $j\in\N$ with $j\leq n$. Using that  $\kappa$  satisfies (i) and that $\kappa$ and $t^{\a}$ are nondecreasing, we see that
\begin{align*}
 n \frac{\kappa(A^n y)}{A^{n \a} }\leq& \sum_{j=1}^n \frac{\kappa(A^j y)}{A^{j \a} } = \frac{A^{\a}}{\log A} \sum_{j=1}^n \frac{\kappa(A^{j}y)}{A^{(j+1)\a}} \int^{A^{j+1}}_{A^{j}}\frac{dt}{t} \leq \frac{A^{\a}}{\log A} \sum_{j=1}^n \int^{A^{j+1}}_{A^{j}} \frac{\kappa (yt)}{t^{1+\a}} dt\\
 \leq& \frac{A^{\a}}{\log A} \int_{1}^{\oo} \frac{\kappa (yt)}{t^{1+\a}} dt \leq B \kappa (y),
\end{align*}
for $y\geq y_0$ large enough for some suitable constant $B>0$. Hence for any $\ep\in (0, 1/A]$ there exists $n$ such that 
$\ep \in (1/A^{n+1},1/A^{n}]$, so for $s\geq A^{n+1} y_0$ we observe that
$$\frac{\ep^\a \kappa (s)}{\kappa (\ep s) } \leq \frac{A^{-n\a} \kappa(s)}{\kappa (A^{-(n+1)}s) }\leq \frac{A^\a B}{n+1},$$
which proves that $\kappa$ satisfies (iii). \para

(iii) $\Rightarrow$ (iv) Immediate.\para
(iv) $\Rightarrow$ (v) Immediate.\para
(v) $\Rightarrow$ (vi) Deduced from Lemma~\ref{lemma.alpha.gamma}.\para
(vi) $\Rightarrow$ (vii) By Theorem~\ref{th:almostmonotonerep.alpha.beta} taking $\a(\sigma)<\ga<\a$. \para
(vii) $\Rightarrow$ (viii) By (vii) and suitably enlarging the constant when $a=0$, we see that  $\displaystyle t \mapsto t^\ga /\sigma(t)$ is almost increasing in $[a,\oo)$. Hence for all $y\ge a$
 $$ y^{-\a} \int_a^{y}\frac{t^{\a}}{\sigma(t)}\frac{dt}{t} = y^{-\a} \int_a^{y}\frac{t^{\a-\ga} t^\ga}{\sigma(t)}\frac{dt}{t} \le M \frac{y^{\ga-\a}}{\sigma (y)} \int_0^{y} t^{\a-\ga} \frac{dt}{t} =  \frac{M}{(\a-\ga)\sigma(y)}. $$
(viii) $\Rightarrow$ (ix)  Since  $\sigma$ is nondecreasing, for every $k\in\N_0=\N\cup\{0\}$ with $k\geq a$ we see that 
$$\int_{k}^{k+1} \frac{t^\a}{\sigma(t)} \frac{dt}{t} \ge \frac{1}{\sigma(k+1)} \frac{(k+1)^\a-k^\a}{\a} = \frac{(k+1)^{\a-1}}{\a \sigma(k+1)} (k+1) \left(1-\left(\frac{k}{k+1}\right)^\a\right).  $$
We write $a_k= (k+1) (1-(k/k+1)^\a)$ for all $k\in\N_0$ and we see that $(a_k)_{k\in\N_0}$ is a sequence of positive real numbers
with $\lim_{k\to\oo} a_k =\a$, so we fix $0<D<\min_{k\in\N_0}(a_k)$. Hence by (viii) for $y=p\geq \lceil a \rceil +1$, we observe that
$$C \frac{p^\a}{\sigma(p)}\ge \int_{a}^p \frac{t^{\a}}{\sigma(t)}\frac{dt}{t}  \geq \sum_{k=\lceil a \rceil}^{p-1}  \int_{k}^{k+1} \frac{t^{\a}}{\sigma(t)}\frac{dt}{t} \geq \frac{D}{\a} \sum_{k=\lceil a \rceil}^{p-1} \frac{(k+1)^{\a-1}}{ \sigma(k+1)}. $$\para

(ix) $\Rightarrow$ (x)  For simplicity we write $k_0:=\lceil a \rceil+1$. First we show that the sequence $(p^\a/\sigma(p))^\oo_{p=k_0 }$ is almost increasing. By (ix) and the monotonicity of $\sigma$, for any $q\geq p\geq k_0+3$ 
we observe that 
\begin{equation*}\label{eq:almost.incresing.intermidiate.step}
 C \frac{ q^\a}{\sigma(q)}\geq \sum_{k=k_0}^{q} \frac{k^{\a-1}}{\sigma(k)}\geq  \sum_{k=k_0}^{p} \frac{k^{\a-1}}{\sigma(k)} \geq \frac{1}{\sigma(p)} \int_{k_0+1}^{p-1}t^{\a}\frac{dt}{t} \geq \frac{p^\a}{\a \sigma(p)} \left(\frac{p-1}{ p}\right)^\a \left(1- \frac{(k_0+1)^\a}{ (k_0+2)^\a}\right).  
\end{equation*}
 Since $(p-1)/p\geq 1/2$, we deduce that $(p^\a/\sigma(p))^\oo_{p=k_0+3 }$ is almost increasing. 
We conclude that the same holds for $(p^\a/\sigma(p))^\oo_{p=k_0}$ by suitably enlarging the almost monotonicity constant. Now, for every $m\geq k_0$, and every $k\in\N$, $k\geq 2$, we apply (ix) for $p=km$ and the almost monotonicity of   $(p^\a/\sigma(p))^\oo_{p=k_0 }$ and we see that
$$C \frac{ (km)^\a}{\sigma(km)}\geq \sum_{j= k_0}^{km} \frac{j^{\a-1}}{\sigma(j)}\geq  \sum_{j=m}^{km} \frac{j^{\a-1}}{\sigma(j)} \geq B  \frac{m^{\a}}{\sigma(m)} \int_{m}^{km} \frac{dt}{t} = B  \frac{m^{\a}}{\sigma(m)} \log(k).   $$
Given $\theta \in (0,1)$ we take $k\in\N$, $k\geq  \max(2,e^{C/(B\theta)} )$  and (x) holds. \para

(x) $\Rightarrow$ (xi) Immediate.\para

(xi) $\Rightarrow$ (xii) 
By (xi), there exist $\theta\in(0,1)$ and $p_0\in\N$  such that $\sigma (k^{s} p)\leq ( \theta k^\a)^s \sigma(p)$ for all $p\geq p_0$ and every $s\in\N_0$. We prove (xii) for $p\geq p_0$, then we conclude by suitably enlarging the constant. By the monotonicity of $\sigma$, for  all $p\geq p_0$ 
 $$\sum_{\ell=p+1}^\oo \frac{\sigma(\ell)}{\ell^{1+\a}} =\sum_{s=0}^\oo \sum_{\ell=k^s p+1}^{k^{s+1} p}\frac{\sigma(\ell)}{\ell^{1+\a}} \le
 \sum_{s=0}^\oo \sigma (k^{s+1}p) \int_{k^{s}p}^{k^{s+1}p} \frac{dt}{t^{1+\a}}.   $$
 Then 
 $$\sum_{\ell=p+1}^\oo \frac{\sigma(\ell)}{\ell^{1+\a}}  \le
 \sigma(p) \sum_{s=0}^\oo \frac{\theta^{s+1}k^{(s+1)\a}}{(k^{s} p)^\a} \int_{k^{s}p}^{k^{s+1}p} \frac{dt}{t} \le  \theta \log( k)  k^\a \frac{\sigma(p)}{p^\a} \sum_{s=0}^\oo \theta^{s} = \frac{\theta \log( k)  k^\a }{1-\theta}\frac{\sigma(p)}{p^\a}.   $$
Consequently, for all $p\geq p_0$ 
$$\sum_{\ell=p}^\oo \frac{\sigma(\ell)}{\ell^{1+\a}} \leq \left(\frac{\theta \log( k)  k^\a }{1-\theta} +1\right)\frac{\sigma(p)}{p^\a}.   $$
By suitably enlarging the constant we see that (xii) holds for all $p\in\N$  with $p\geq \lceil a \rceil $.
\para 

(xii) $\Rightarrow$ (i)
For all $y\geq 1$
by the monotonicity of $\sigma $
 $$\int_y^{\infty}\frac{\sigma(t)}{t^{1+\a}}dt \leq \sum_{k=\lfloor y \rfloor}^\oo \int_k^{k+1} \frac{\sigma(t)}{t^{1+\a}}dt \leq  \sum_{k=\lfloor y \rfloor}^\oo \frac{\sigma (k+1)}{(k+1)^{1+\a}} \left(\frac{k+1}{k}\right)^{1+\a} \leq  2^{1+\a} \sum_{k=\lfloor y \rfloor}^\oo \frac{\sigma (k)}{k^{1+\a}}. $$
Consequently, for all $y\geq y_0:=\max(1,\lceil a \rceil)$ applying (xii) and the monotonicity of $\sigma$
 $$\int_y^{\infty}\frac{\sigma(t)}{t^{1+\a}}dt \leq 2^{1+\a}  C \frac{\sigma(\lfloor y \rfloor)}{\lfloor y \rfloor^\a}   \leq  4^{1+\a} C \frac{\sigma( y )}{ y ^\a}.    $$
Then, (i) holds for $y\geq y_0$.   By suitably modifying the constant 
it also holds for $y\in(0,y_0)$, because
$$y^{\a} \int_y^{\infty}\frac{\sigma(t)}{t^{1+\a}}dt =  y^{\a} \left( \int_{y}^{y_0}\frac{\sigma(t)}{t^{1+\a}}dt + \int_{y_0}^{\infty}\frac{\sigma(t)}{t^{1+\a}}dt\right) \leq \frac{\sigma(y_0)}{\a} + 4^{1+\a} C \sigma( y_0 ) .$$


\end{proof}

\begin{remark}\label{remark:alpha.stability.open}
 The value of $\a$ is stable for $\hyperlink{sim}{\sim}$ because the equivalence implies \eqref{eq:equiv.functionsORV}  for nondecreasing functions tending to infinity at infinity, so all the conditions above are stable under equivalence for nondecreasing functions tending to infinity. 
 \para
 
 The list of equivalent definitions of $\a(\sigma)$ given in \cite{AljAra77}  can be increased, we write
 $$\a(\sigma)=\inf\{\a>0;\, \text{any of the conditions in Theorem~\ref{th:main.th.alpha} holds for}\,\, \sigma\},$$
 if the previous set is empty we write $\a(\sigma)=0$. It is worthy to notice that if any of the statements in the previous theorem holds for $\a$ it is also valid for $\a-\ep$ for some small enough $\ep>0$. Hence the set of values satisfying any these conditions is  left-open.\para
 
\end{remark}

%
%
%
%
%
%
%

Since for $\a=1$ condition (i) in the previous theorem is  \hyperlink{omsnq}{$(\omega_{\on{snq}})$} condition, we obtain the following corollary.

\begin{corollary}\label{coro:omsnq}
Let $\sigma$ be as above. The following are equivalent:
\begin{itemize}
\item[(i)] $\sigma$ satisfies \hyperlink{omsnq}{$(\omega_{\on{snq}})$},
\item[(ii)] $\sigma$ satisfies every/some of the equivalent conditions (ii)--(xii) in Theorem~\ref{th:main.th.alpha} for $\a=1$.
\end{itemize}
In particular,  $\sigma$ satisfies \hyperlink{omsnq}{$(\omega_{\on{snq}})$} if and only if $\ga(\sigma)>1$  if and only if $\a(\sigma)<1$.
\end{corollary}


Similarly, we can translate condition \hyperlink{om1}{$(\omega_{1})$} in terms of the index $\a$. 
The following corollary can be seen as a restatement of a well-known result of W. Feller (e.g. see \cite[Coro. 2.0.6]{BingGoldTeug89}).

\begin{corollary}\label{coro:om1}
Let $\sigma$ be as above. The following are equivalent:
\begin{itemize}
\item[(i)] $\sigma$ satisfies \hyperlink{om1}{$(\omega_{1})$},
\item[(ii)] There exists $\a>0$ such that $\sigma$ satisfies every/some of the equivalent conditions (i)--(xii) in Theorem~\ref{th:main.th.alpha}.
\end{itemize}
In particular, $\sigma$ satisfies {\hyperlink{om1}{$(\omega_{1})$}} if and only if $\ga(\sigma)>0$  if and only if $\a(\sigma)<\oo$.
\end{corollary}
\begin{proof}
 It is immediate to check that, if $\sigma$ satisfies \hyperlink{om1}{$(\omega_{1})$}, then it exists $\a>0$ such that Theorem~\ref{th:main.th.alpha}.(iv) holds for $K=2$. If Theorem~\ref{th:main.th.alpha}.(iv) holds for some $K\geq 2$, then $\sigma$ satisfies \hyperlink{om1}{$(\omega_{1})$} by monotonicity. If Theorem~\ref{th:main.th.alpha}.(iv) holds for some $1<K<2$,  we fix $n\in\N$ such that $2\leq K^n$, by monotonicity, we observe that 
 $$\limsup_{t\to\oo} \frac{\sigma(2t)}{\sigma (t)} \leq \limsup_{t\to\oo} \frac{\sigma(K^n t)}{\sigma (t)} \leq
 \limsup_{t\to\oo} \frac{\sigma(K^n t)}{\sigma (K^{n-1}t)} \cdots  \frac{\sigma(Kt)}{\sigma (t)} < K^{n\a}.  $$
 Hence $\sigma$ satisfies \hyperlink{om1}{$(\omega_{1})$}. 
\end{proof}

Since $\sigma$ is nondecreasing $\b(\sigma)\geq 0$ and so, according to Theorem~\ref{th:Ch.ORVF},  $\sigma\in ORV$ if and only if $\a(\sigma)<\oo$  if and only if $\gamma(\sigma)>0$  if and only if $\sigma$ satisfies {\hyperlink{om1}{$(\omega_{1})$}}.\parb

\begin{remark}\label{rema:om2om5omnq}
 Conditions \hyperlink{om2}{$(\omega_2)$}, \hyperlink{om5}{$(\omega_5)$} and {\hyperlink{omnq}{$(\omega_{\text{nq}})$}} are  instead connected to the order $\ro(\sigma)$. For $\sigma$ as above, thanks to the relation between $\a(\sigma)$ and $\ro(\sigma)$ (see Remark~\ref{rema:ordersROMU}) we see that each assertion implies the following:
\begin{enumerate}[(i)]
 \item $\a(\sigma)<1$,
 \item $\ro(\sigma)<1$,
 \item there exists $\alpha\in(0,1)$ such that $\omega(t)=O(t^{\alpha})$ as $t\rightarrow\infty$,
 \item $\sigma$ satisfies \hyperlink{omnq}{$(\omega_{\text{nq}})$},
 \item $\sigma$ satisfies \hyperlink{om5}{$(\omega_{5})$},
 \item $\sigma$ satisfies \hyperlink{om2}{$(\omega_{2})$},
 \item $\ro(\sigma)\leq 1$,
\end{enumerate}
and only the implication (ii) $\Rightarrow$ (iii) can be reversed. Hence if $\sigma$ satisfies \hyperlink{omsnq}{$(\omega_{\text{snq}})$}, applying Corollary~\ref{coro:omsnq}, we see that $\sigma$ satisfies the conditions (i)-(vii) and, by Corollary~\ref{coro:om1}, $\sigma$ has also \hyperlink{om1}{$(\omega_1)$}. Part of this information was well-known but dispersed, see \cite[Coro. 1.4]{MeiseTaylor88} and the main novelty is its connection to O-regular variation.\para
\end{remark}

Finally, the growth condition \hyperlink{om6}{$(\omega_6)$} is associated with the lower Matuszewska index $\b(\sigma)$ and it is possible to establish a result analogous to Theorem~\ref{th:main.th.alpha}. 

\begin{theorem}\label{th:main.th.beta}
Let $\sigma$ be as above and $\b\geq 0$. We take $a\geq 0$ such that $\sigma(x)>0$ for all $x\ge a$ and  $\sigma(x)=0$ for every $x\leq\lceil a \rceil-1$. The following are equivalent:
\begin{itemize}
\item[(i)]  there exists $C>0$ such that  $\displaystyle  \int_1^{y}\frac{\sigma(t) }{t^{\b+1}}dt\le C \frac{\sigma(y)}{y^\b}$ for all $y\geq 1$,

\item[(ii)] $\displaystyle \lim_{k\rightarrow \oo}\liminf_{t\rightarrow\infty}\frac{\sigma(k t)}{k^\b \sigma( t)}=\oo$, 
\item[(iii)] there exists $K>1$  such that $\displaystyle \liminf_{t\rightarrow\infty}\frac{\sigma(Kt)}{\sigma(t)}> K^\b$,
\item[(iv)] $\b(\sigma)>\b$  (with the convention in Remark~\ref{rema:Matuszewska.indices.welldef}),
\item[(v)] there exists $\ga>\b$ such that $\displaystyle t\mapsto \frac{\sigma(t)}{t^\ga}$ is almost increasing  in $[a,\oo)$ if $a>0$ and in $[\ep,\oo)$ for all $\ep>0$ if $a=0$,
\item[(vi)] there exists $C>0$ such that $\displaystyle\frac{1}{y^{\b}} \int_{y}^\oo \frac{t^{\b-1}}{\sigma(t) }dt\le \frac{C}{\sigma(y)}$  for all  $y\geq  a$ if $a>0$ and for all $y\geq \ep$ if $a=0$ where $\ep>0$ is arbitrary but fixed and $C$ depends on $\ep$,
\item[(vii)]there exists $C>0$ such that  
$\displaystyle \sum_{k=p}^\oo \frac{k^{\b-1}}{\sigma(k)}\leq C \frac{ p^{\b}}{\sigma(p)}$ for every $p\in\N$ with $p\geq a$,
\item[(viii)] for every $\theta\in(0,1)$ there exists $k\in\N$, $k\geq 2$, such that 
$\sigma (p) \leq \theta k^{-\b}  \sigma \left(k p \right)$ for every $p\in\N$,
\item[(ix)] there exists $k\in\N$, $k\geq 2$,  such that $\displaystyle \liminf_{p\in\N,\, p\rightarrow\infty}\frac{\sigma(kp)}{\sigma(p)}> k^\b$,
\item[(x)] there exists $C>0$ such that  
$\displaystyle \sum_{k=1}^p \frac{\sigma(k)}{k^{1+\b}}\leq C \frac{\sigma(p)}{p^\b}$ for every $p\in\N$. 
\end{itemize}

\end{theorem}

\begin{proof}
 (i) $\Rightarrow$ (ii) First, we assume that $\b=0$. By (i) and the monotonicity of $\sigma$ for all $y\geq 1$ and every $k>1$ we see that
 $$ C\sigma(ky)\geq \int_1^{ky}\frac{\sigma(t) }{t}dt \geq \int_y^{ky}\frac{\sigma(t) }{t}dt \ge \sigma (y)  \log(k). $$
Then (ii) holds for $\b=0$. Secondly, if $\b>0$, applying (i) twice and using the monotonicity of $\sigma$ for all $y\geq 1$ and every $k>1$ we observe that
\begin{align*}
 C^2 \frac{\sigma(ky)}{(ky)^\b} &\geq  \int_y^{ky} C \frac{\sigma(t) }{t^{\b}}\frac{dt}{t} \geq
 \int_y^{ky} \int_y^{t} \frac{\sigma(u) }{u^{\b}}\frac{du}{u}\frac{dt}{t} \geq \sigma(y)
 \int_y^{ky} \left[ \frac{1 }{\b y^{\b}} -\frac{1}{\b t^\b} \right] \frac{dt}{t}\\ &= \frac{\sigma(y)}{\b y^\b}
 \left[ \log(k) - y^\b \left(  \frac{1 }{\b y^{\b}} - \frac{1}{\b (ky)^\b} \right) \right] = 
 \frac{\sigma(y)}{\b^2 y^\b} \left[\b \log(k) - 1 + k^{-\b} \right].
 \end{align*}
 Since $\lim_{k\to\oo}  \left[\b \log(k) - 1 + k^{-\b} \right]= \oo$, (ii) holds.\para 
  (ii) $\Rightarrow$ (iii) Immediate.\para
  (iii) $\Rightarrow$ (iv) There exists $K>1$ such that $\sigma_{\on{\low}} (K)>K^\b$. Since $\sigma$ is nondecreasing one may apply Theorem~\ref{th:convenient.rep.alpha.beta} and deduce that (iv) holds. \para
  (iv) $\Rightarrow$ (v) Immediate by Theorem~\ref{th:almostmonotonerep.alpha.beta}.\para
  (v) $\Rightarrow$ (vi) Analogous to (vii) $\Rightarrow$ (viii) in Theorem~\ref{th:main.th.alpha}.\para
	
(vi) $\Rightarrow$ (vii) If $\b>0$, it is analogous to (viii) $\Rightarrow$ (ix) in Theorem~\ref{th:main.th.alpha}. If $\b=0$, we use that for every $k\ge a$
$$\int_{k}^{k+1} \frac{1}{\sigma(t)} \frac{dt}{t} \ge \frac{1}{(k+1)\sigma(k+1)}, $$
and we conclude as for $\b>0$.\para

(vii) $\Rightarrow$ (viii) Applying condition (vii) twice and using the monotonicity of $\sigma$, for all $p\in\N$ with $p\geq a$ and every $m\geq 2$ we see that
\begin{align*}
 C^2 \frac{p^\b}{\sigma(p)}\geq&  \sum^\oo_{k=p} C \frac{k^{\b-1}}{\sigma(k)} \geq 
 \sum^\oo_{k=p} \frac{1}{k} \sum^\oo_{\ell= k}  \frac{\ell^{\b-1}}{\sigma(\ell)} \geq
 \sum^{mp-1}_{k=p} \frac{1}{k} \sum^{mp}_{\ell= k}  \frac{\ell^{\b-1}}{\sigma(\ell)} \geq 
\sum^{mp-1}_{k=p} \frac{1}{k mp\,\sigma (mp)} \sum^{mp}_{\ell= k}  \ell^{\b} \\
\geq&  \frac{1}{mp\,\sigma (mp)} \sum^{mp-1}_{k=p} \frac{1}{k } \int^{mp}_{k}  t^{\b} dt
= \frac{(mp)^{\b}}{\sigma (mp) (\b+1)} \sum^{mp-1}_{k=p} \left[\frac{1}{k} -\frac{k^\b}{(mp)^{\b+1}} \right]\\ \geq&
\frac{(mp)^{\b}}{\sigma (mp) (\b+1)} \int_p^{mp} \left[\frac{1}{t} -\frac{t^\b}{(mp)^{\b+1}} \right] dt =
\frac{(mp)^{\b}}{\sigma (mp) (\b+1)^2} [\log(m^{\b+1}) - 1 +\frac{1}{m^{\b+1}} ].
 \end{align*}
Then, given $\theta\in(0,1)$, we take $m$ large enough  such that
$$\frac{\sigma(mp) }{m^\b \sigma(p)}\geq \frac{1}{C^2(\b+1)^2}[ \log(m^ {\b+1}) - 1 +\frac{1}{m^{\b+1}}] \geq \frac{1}{\theta},$$
for all $p\in\N$ with $p\geq a$. For $1\leq p<a$, $\sigma(p)=0$ and (viii) trivially holds.\para

(viii) $\Rightarrow$ (ix) Immediate.\para
(ix) $\Rightarrow$ (x)
 First, we prove that there exists $\ep\in(0,1)$ such that the sequence $(\sigma(p)/p^{\b+\ep})^\oo_{p=1}$ is almost increasing.  By (ix), there exists $\ep\in (0,1)$ and $p_0\in\N$ such that for every $p\geq p_0$, $\sigma(kp)> \sigma(p) k^{\b+\ep}$. Using the monotonicity of $\sigma$, we see that for every $p,q\in\N$ with $q\geq p\geq p_0$ there exists $s\in\N_0$ with $k^s p \leq q < k^{s+1} p$ and we observe that
$$\frac{\sigma(q)}{q^{\b+\ep}}\geq \frac{\sigma(k^s p)}{(k^{s+1}p)^{\b+\ep}}\geq \frac{ (k^{\b+\ep} )^s \sigma(p)}{ (k^{s+1}p)^{\b+\ep}}\geq
\frac{1}{k^{\b+\ep}} \frac{ \sigma(p)}{ p^{\b+\ep}}.$$
By the monotonicity of $\sigma$, we conclude that $(\sigma(p)/p^{\b+\ep})^\oo_{p=1}$ is almost increasing.
 We use this property to show that for all $p\in\N$
$$\sum_{k=1}^p \frac{\sigma(k)}{k^{\b+1}} \leq C \frac{\sigma(p)}{p^{\b+\ep}} \sum_{k=1}^p  \frac{1}{k^{1-\ep}} \leq C \frac{\sigma(p)}{p^{\b+\ep}} \int^p_{0}\frac{dt}{t^{1-\ep}} = \frac{C}{\ep} \frac{\sigma(p)}{p^{\b}}. $$\para

(x) $\Rightarrow$ (i) By the monotonicity of $\sigma$, for $y\geq 1$ we see that
$$\int_{1}^y \frac{\sigma(t)}{t^{1+\b}} dt \leq \sum_{k=1}^{\lfloor y\rfloor-1} \sigma(k+1) \int_{k}^{k+1} \frac{dt}{t^{1+\b}} + \sigma (y)\int_{\lfloor y\rfloor}^y \frac{dt}{t^{1+\b}} \leq \sum_{k=1}^{\lfloor y\rfloor-1} \frac{\sigma(k+1)}{k^{1+\b}}  +  \frac{\sigma (y)}{ (\lfloor y\rfloor)^{1+\b} },  $$
where the sum does not appear if $\lfloor y\rfloor=1$. Using that  $(k+1)/k \leq 2$ and $y/\lfloor y\rfloor\leq 2$ for all $y,k\geq 1$ and (x) we observe that 
$$\int_{1}^y \frac{\sigma(t)}{t^{1+\b}} dt  \leq 2^{1+\b}
\sum_{k=1}^{\lfloor y\rfloor-1} \frac{\sigma(k+1)}{(k+1)^{1+\b}}  +  2^{\b} \frac{\sigma (y)}{ y^{\b} }
\leq 2^{1+\b} C \frac{\sigma(\lfloor y\rfloor)}{(\lfloor y\rfloor)^{\b} } +  2^{\b} \frac{\sigma (y)}{ y^{\b} }
\leq (2^{1+2\b}C+2^\b) \frac{\sigma (y)}{ y^{\b} }.$$
\end{proof}

The index $\ga(\sigma)$ naturally appears in the study of the surjectivity of the Borel map in ultraholomorphic classes, see
~\cite{JimenezSanzSchindlLaplace}.
According to the last result,
one might analogously define an index
\begin{equation}\label{lowernewindex2}
\overline{\gamma}(\sigma):=\inf\{\gamma>0: \exists\;A>1:\;\;\;\liminf_{t\rightarrow\infty} \frac{\sigma(A^{\gamma}t)}{\sigma(t)}>A\}.
\end{equation}
We have refrained from providing its detailed study, similar to that of $\gamma(\sigma)$, but we can mention that
$1/\overline{\gamma}(\sigma)=\beta(\sigma)$, and that $\sigma$ has \hyperlink{om6}{$(\omega_6)$} if and only if $\overline{\gamma}(\sigma)$ is finite.
If $\b>0$, $\sigma$ is as above and in addition $\a(\sigma)<\oo$, then one can show that $\b(\sigma)>\b>0$ if and only if there exists a nondecreasing function $\kappa: [1,+\infty)\longrightarrow[0,+\infty)$ such that $\sigma {\hyperlink{sim}\sim}\kappa$, $\kappa$ satisfies Theorem~\ref{th:main.th.beta}.(i) and $\kappa(t^{1/\b})$ is convex, recovering the missing equivalent condition which is available for $\a(\sigma)$ but is not for $\b(\sigma)$ in general. Finally, the considerations made in Remark~\ref{remark:alpha.stability.open} for $\a(\sigma)$ are also valid for $\b(\sigma)$.  \para

Using that $\lim_{t\to\oo}\sigma(t)=\oo$, with a proof similar to the one of Corollary~\ref{coro:om1}, we see that 
Theorem~\ref{th:main.th.beta}.(iii) is satisfied for $\b=0$ if and only if $\sigma$ satisfies \hyperlink{om6}{$(\omega_{\on{6}})$} and we obtain the desired characterization of this growth property.  

\begin{corollary}\label{coro:om6}
Let $\sigma$ be as above. The following are equivalent:
\begin{itemize}
\item[(i)] $\sigma$ satisfies \hyperlink{om6}{$(\omega_{\on{6}})$},
\item[(ii)]$\sigma$ satisfies every/some of the equivalent conditions (i)--(x) in Theorem~\ref{th:main.th.beta} for $\b=0$.
\end{itemize}
In particular,  $\sigma$ satisfies \hyperlink{om6}{$(\omega_{\on{6}})$} if and only if  $\b(\o)>0$.
\end{corollary}

\begin{remark}\label{rema:RVandproperties}
According to Remark~\ref{rema:ordersROMU}, it is worthy to notice that, if $\sigma$ is of regular variation of index $\o$, what happens for most of the examples appearing in the applications, then 
all the information provided by the previous results is concentrated in one value since $\b(\sigma)= \mu(\sigma)=\ro(\sigma)=\a(\sigma)=\o\in[0,\oo)$. In this case, $\sigma$ always satisfies \hyperlink{om1}{$(\omega_{\on{1}})$}, and $\o<1$ if and only if $\sigma$ satisfies \hyperlink{omsnq}{$(\omega_{\text{snq}})$}. If $\o=0$, $\sigma$ is said to be of {\it slow variation} and we see that $\sigma$, regularly varying, does not satisfy \hyperlink{om6}{$(\omega_{\on{6}})$} if and only if it is of slow variation. 
\end{remark}

\begin{remark}\label{rema:weightsAT0}
The results presented in this subsection are prepared to be applied to weight functions whose mass is concentrated at $\oo$. However, in the literature of weighted spaces it is common to find a function $h:(0,\oo)\to(0,\oo)$ nonincreasing with $\lim_{t\to 0} h(t)=\oo$ or a function $H:(0,\oo)\to(0,\oo)$ nondecreasing with $\lim_{t\to 0} H(t)=0$ as weight functions for the corresponding structure. Since in that context similar conditions appear for such functions, see~\cite[(2.1) and (2.2)]{Blasco12} or the definition of regular modulus of continuity in~\cite{Dynkin80}, one might be tempted to obtain an analogous version of Theorems~\ref{th:main.th.alpha} and~\ref{th:main.th.beta} which could be done by defining $\sigma(t)=h(1/t)$ in the first case and $\sigma(t)=1/H(1/t)$ in the second one.  
\end{remark}

In the following examples we can compute the indices and deduce the corresponding properties for $\sigma$ according to the previous corollaries:
\begin{itemize}
\item[(i)] The Gevrey weights $\omega(t)=t^s$, $0<s\le 1$, are regularly varying of index $\ro=\a(\o)=\ro(\o)=\mu(\o)=\b(\o)=s$ and $\ga(\o)=1/s$. 

\item[(ii)] The weights $\omega(t)=t/(\log(e+t))^{\alpha}$, $\alpha\in\RR$,  are also regularly varying of index $\ro=\a(\o)=\ro(\o)=\mu(\o)=\b(\o)=\ga(\o)=1$. 

\item[(iii)] The weights $\omega(t)=\max\{0,\log(t)^s\}$, $s>1$, are regularly varying, in fact they are slowly varing so
$\ro=\a(\o)=\ro(\o)=\mu(\o)=\b(\o)=0$ and $\ga(\o)=\oo$. 
\end{itemize}

Moreover, for any $\sigma$ as above and any $s>0$, by Remark~\ref{rema.properties.Matuszewska.indices}, we observe that $\a(\sigma)<1/s$ (or resp. $\b(\sigma)>1/s$) if and only if some/every of the conditions in Theorem~\ref{th:main.th.alpha} (or resp. some/every of the conditions in Theorem~\ref{th:main.th.beta}) holds for $\sigma_s(t)=\sigma(t^s)$ and $\a=1$ (or resp. $\b=1$). Hence the list of equivalent definitions of the indicies $\a$ and $\b$ might be increased. In particular, we see that $\sigma$ satisfies \hyperlink{om1}{$(\omega_1)$} (or respectively \hyperlink{om6}{$(\omega_6)$}) if and only if $\sigma_s$  satisfies \hyperlink{om1}{$(\omega_1)$} (or resp. \hyperlink{om6}{$(\omega_6)$}) and we also see that $\a(\sigma)<1/s$ if and only if  $\sigma_s$ satisfies \hyperlink{omsnq}{$(\omega_{\on{snq}})$}. The same is valid if $\sigma_s$ is replaced by  $\sigma^s(t)= (\sigma(t))^s$. \para

It is worthy to notice that conditions \hyperlink{om3}{$(\omega_3)$} and \hyperlink{om4}{$(\omega_4)$} are related to 
the {\itshape Legendre-Fenchel-Young-conjugate} $\varphi^{*}_{\sigma}(x):=\sup\{x y-\sigma(e^y): y\ge 0\},$ defined for $x\geq 0$ and they seem not to be connected to O-regular variation.\para

Finally,  for condition \hyperlink{om7}{$(\omega_7)$}, introduced and described in~\cite[Lemmas 3.6.1, 5.4.1]{Schindldissertation}, \cite[Lemma 5.9 (5.12)]{RainerSchindlcomposition}  and~\cite[Appendix A]{JimenezSanzSchindlExtensionReal} and for $\sigma$ as above satisfying this condition, we can show that $\a(\sigma)=0$ as follows: If $\sigma$ satisfies \hyperlink{om7}{$(\omega_7)$} there exist $H, C, t_0\geq 0$ such that for all $t\ge t_0$,  $\sigma(t^2)\le C\sigma(H t)$, so for every $\lambda\in(1,\oo)$ there exists $t_\lambda \ge H t_0$ such that for all $t\ge t_\lambda$, $\sigma(\lambda t)\leq \sigma(t^2/H^2)\leq C \sigma(t) $. Hence $\sigma^{\on{up}} (\lambda)\leq C$ for every $\lambda\in(1,\oo)$. This implies the desired property by Theorem~\ref{th:convenient.rep.alpha.beta}, so $\b(\sigma)=\mu(\sigma)=\ro(\sigma)=\a(\sigma)=0$ and we recover the well-known incompatibility between \hyperlink{om7}{$(\omega_7)$} and \hyperlink{om6}{$(\omega_6)$}. Note that \hyperlink{om7}{$(\omega_7)$} is just a sufficient condition for having $\gamma(\sigma)=+\infty$.

\subsection{Legendre conjugates and the index \texorpdfstring{$\ga$}{Gamma}}\label{subsect.LegendreconjugateIndex}

As it was pointed out in Remark~\ref{rema.properties.Matuszewska.indices}, for any $\sigma$ we observe that
$$\b(\sigma(t) t) =\b(\sigma(t))+1, \qquad  \b(\sigma(t)) =\b(\sigma(t)/t)+1,$$
and the same holds for $\a$. Motivated by the relation between weight sequences and weight functions described in Section~\ref{AssociatedfunctionsORV}
and by its necessity for  the applications, see~\cite{JimenezSanzSchindlLaplace}, we want to obtain some similar relation for the index $\ga$, that is, we look for functions $\kappa_1,\kappa_2$ such that
\begin{equation}\label{eq:kappa1.kappa2}
 \ga (\sigma)= \ga(\kappa_1) +1, \qquad \ga(\kappa_2)=\ga(\sigma)+1.
\end{equation}
For this purpose we consider the Legendre conjugates. Let $\sigma:[0,\infty)\longrightarrow[0,\infty)$ be nondecreasing with $\lim_{t\rightarrow\infty}\sigma(t)=\infty$, 
then for any $s\ge 0$ we define the so-called {\itshape upper Legendre conjugate (or upper Legendre envelope)} of $\sigma$ by
\begin{equation*}\label{omegaconjugate}
\sigma^{\star}(s):=\sup_{t\ge 0}\{\sigma(t)-st\}.
\end{equation*}

We summarize some basic properties, see \cite[Remark 1.5]{PetzscheVogt} and \cite[(8), p. 156]{Beurling72}. 
 By definition, $\sigma^{\star}(0)=\infty$. If $\sigma$ has in addition \hyperlink{om5}{$(\omega_5)$}, then $\sigma^{\star}(s)<\infty$ for all $s>0$. In this case, the function $\sigma^{\star}:(0,\infty)\rightarrow[0,\infty)$ is nonincreasing,  convex and continuous with $\lim_{s\to 0}\sigma^{\star}(s)=\infty$ and   $\lim_{s\to \infty}\sigma^{\star}(s)=\lim_{t\to0}\sigma(t)$.\para

On the other hand for any $h:(0,\infty)\rightarrow[0,\infty)$ which is nonincreasing and $\lim_{t\rightarrow 0}h(t)=\infty$ for $t\geq 0$ we can define the so-called {\itshape lower Legendre conjugate (or envelope)}  by
\begin{equation*}\label{omegaconjugate0}
h_{\star}(t):=\inf_{s>0}\{h(s)+ts\}.
\end{equation*}
We observe that $h_{\star}$ is nondecreasing, concave and continuous with $\lim_{t\rightarrow\infty}h_{\star}(t)=\infty$ and 
$\lim_{t\rightarrow 0}h_{\star}(t)=\lim_{s\rightarrow \oo}h(s)$.\para
\begin{remark}\label{rema:majorantminorant.Legendre}
 Let $\sigma:[0,\infty)\longrightarrow[0,\infty)$ be nondecreasing with $\lim_{t\rightarrow\infty}\sigma(t)=\infty$ satisfying \hyperlink{om5}{$(\omega_5)$} and $h:(0,\infty)\rightarrow[0,\infty)$ be nonincreasing with $\lim_{t\rightarrow 0}h(t)=\infty$.  We observe that
\begin{enumerate}[(i)] 
 \item The function $(\sigma^{\star})_{\star}: (0,\oo) \to (0,\oo)$ is nondecreasing, concave and continuous with $\lim_{t\rightarrow\infty}(\sigma^{\star})_{\star}(t)=\infty$ and it is indeed the least concave majorant of $\sigma$ (in the sense that, if $\tau:[0,+\infty)\rightarrow[0,+\infty)$ is concave and $\sigma\le\tau$, then $(\sigma^{\star})_{\star}\le\tau$).
 \item The function $(h_{\star})^{\star}: (0,\oo) \to (0,\oo)$ is well-defined since a direct computation leads to $(h_{\star})^{\star} (u) \leq h (u) $ for all $u>0$, so $h_{\star}$ does not need to satisfy \hyperlink{om5}{$(\omega_5)$}. Moreover, $(h_{\star})^{\star}$ is nonincreasing, convex and continuous with $\lim_{t\rightarrow 0}h_{\star} (t)=\infty$,
 it is indeed the largest convex minorant of $h$ (in the sense that, if $k:[0,+\infty)\rightarrow[0,+\infty)$ is convex and $h\ge k$, then $(h_{\star})^{\star}\ge k$). 
\end{enumerate}
\end{remark}

When we consider the upper and lower Legendre conjugates the information is transferred from $0$ to $\oo$ and vice versa, so for any positive function $f$ defined in an interval $I\en (0,\oo)$ it is helpful to introduce the function $f^{\iota}(t):=f(1/t)$ in the corresponding  subinterval of $(0,\oo)$. The first result compares the indices of $\sigma$ and $(\sigma^{\star})^\iota$.

 \begin{proposition}\label{prop:indices.upperconjugate}
 Let $\sigma:[0,\infty)\longrightarrow[0,\infty)$ be nondecreasing with $\lim_{t\rightarrow\infty}\sigma(t)=\infty$. Assume that $\sigma$ satisfies \hyperlink{om5}{$(\omega_{\on{5}})$} and that $\sigma$ is equivalent to its least concave majorant, i.e., $\sigma \sim (\sigma^{\star})_{\star}$. Then 
 \begin{equation}\label{eq:gamma.upperconjugate}
  \ga(\sigma)=\ga((\sigma^{\star})^{\iota})+1.
 \end{equation}

\end{proposition}
\begin{proof}
We fix $\ga<\ga(\sigma)$, so there exists $K,H>1$ and $\ep\in(0,1)$  such that for all $t\geq 0$
$$\sigma(K^{\ga}t)\leq K^{1-\ep} \sigma(t) +H. $$
Using that $(\sigma^{\star})^{\iota}$ is nondecreasing, for all $s>0$ we see that 
\begin{align*}
 (\sigma^{\star})^{\iota} (K^{\ga-1} s) &= \sup_{t\ge 0} \Big\{ \sigma(t)-\frac{t}{K^{\ga-1} s} \Big\} = 
\sup_{u\ge 0} \Big\{ \sigma(K^\ga u)-\frac{K u}{  s} \Big\} \leq \sup_{u\ge 0} \Big\{ K^{1-\ep} \sigma(u)-\frac{K u}{  s} \Big\} + H 
\\ &\leq K^{1-\ep} (\sigma^{\star})^{\iota} (K^{-\ep} s) +H \leq  K^{1-\ep} (\sigma^{\star})^{\iota} (s) +H.
\end{align*}
Since $\lim_{s\to\oo} (\sigma^{\star})^{\iota} (s) =\oo $, we deduce that $\ga-1< \ga ( (\sigma^{\star})^{\iota})$. \parb

Conversely, we fix $\ga<\ga ( (\sigma^{\star})^{\iota})$, so there exists $K>1$ and $\ep_0\in(0,1)$ such that for all $\ep\in(0,\ep_0)$ there exists $H_\ep>0$  such that for all $t> 0$
$$(\sigma^{\star})^{\iota} (K^{\ga}t)\leq K^{1-\ep} (\sigma^{\star})^{\iota}(t) +H_\ep. $$
Hence, for all $t>0$ we observe that
\begin{align*} 
(\sigma^\star)_\star (K^{\ga+1-\ep} t)\leq K^{1-\ep} (\sigma^\star)_\star (t)+H_\ep.
\end{align*}
Then $\ga+1-\ep<\ga((\sigma^\star)_\star) $ and, since this holds for all $\ep\in(0,\ep_0)$, 
we deduce that $\ga+1\leq \ga((\sigma^\star)_\star)$.
Since the value of the index $\ga$ is stable for \hyperlink{sim}{$\sim$} we conclude that $\ga+1\leq\ga(\sigma)$. 
\end{proof}
 
Hence we have found a candidate for $\kappa_1$ in~\eqref{eq:kappa1.kappa2}, we will show that under suitable assumptions $\kappa_2$ can be chosen as $(\sigma^\iota)_{\star}$.
 
  \begin{proposition}\label{prop:indices.lowerconjugate}
  Let $\sigma$ be as above. Assume that $\sigma^\iota$ is equivalent to its largest convex minorant, i.e., $\sigma^\iota \sim ((\sigma^\iota)_{\star})^{\star}$. Then 
 \begin{equation}\label{eq:gamma.lowerconjugate}
   \gamma(\sigma)+1=\gamma((\sigma^{\iota})_{\star}).
 \end{equation}

\end{proposition}
 
\begin{proof}
We observe that $(\sigma^{\iota})_{\star}:[0,\oo) \to [0,\oo)$ is  nondecreasing with $\lim_{t\to\oo} (\sigma^{\iota})_{\star}(t)=\oo$, concave and continuous. Then $((\sigma^{\iota})_{\star})^{\star}$ is  well-defined and has the properties described in Remark~\ref{rema:majorantminorant.Legendre}. We can follow the proof of Proposition~\ref{prop:indices.upperconjugate} and show that 
$$ \ga((\sigma^{\iota})_{\star})=\ga(((\sigma^{\iota})_{\star})^{\star})^{\iota})+1.$$
Since $\sigma^\iota \sim ((\sigma^\iota)_{\star})^{\star}$, then $\sigma \sim (((\sigma^\iota)_{\star})^{\star})^\iota$ and, by the stability of the index $\ga$ for \hyperlink{sim}{$\sim$}, we conclude that \eqref{eq:gamma.lowerconjugate} is valid.
\end{proof}

If $\sigma \sim \kappa$ with $\kappa$ concave or if $\sigma^\iota \sim \tau$ with $\tau$ convex,  then 
 $\sigma$ is equivalent to its least concave majorant  or $\sigma^\iota$ is equivalent to its largest convex minorant, respectively, and Proposition~\ref{prop:indices.upperconjugate} and~\ref{prop:indices.lowerconjugate} are also valid.
 Nevertheless, it will helpful to see  if these properties  are satisfied under some standard assumption. For $f:(0,\oo)\to (0,\oo)$ J. Peetre~\cite{peetre} shows that
 \begin{equation}\label{eq:concavePeetre}
 f(s)\leq C \max\left(1, \frac{s}{t} \right) f(t)
\end{equation}
is sufficient for $f$ to satisfy $(2C)^{-1} F(x) \leq f(x) \leq F(x)$ for all $x>0$ where $F$ is its least concave majorant. Note that condition~\eqref{eq:concavePeetre} holds if and only if $f$ is almost increasing in $(0,\oo)$ and $f(t)/t$ is almost decreasing in $(0,\oo)$. Since we want to allow the function $f$ to take the value $0$, a suitable modification of  
\eqref{eq:concavePeetre} is needed which entails the appearance of an additional summand which does not destroy the equivalence relation.

\begin{proposition}\label{prop:concavePeetre}
Let $f:[0,\oo)\to [0,\oo)$ be such that there exists $C\geq 1$  with
 \begin{equation}\label{eq:concavePeetre.0.value}
 f(s)\leq C f(t)\qquad \text{and} \qquad f(t)\frac{s}{t}\leq C (f(s)+1),
\end{equation}
for all $t>s\geq 0$. Then there exists $A\geq 1$ such that for every $x\in (0,\oo)$
\begin{equation}\label{eq:concavemajorant.equivalence}
AF(x)-A\leq f(x) \leq F(x),
\end{equation}
where $F$ is the least concave majorant of $f$. 
\end{proposition}
\begin{proof}
 The least concave majorant of $f$ can be represented by
 $$F(x)=\sup\{\lambda_1 f(x_1)+\lambda_2 f(x_2);\, \lambda_1+\lambda_2=1,\,\, \lambda_1x_1+\lambda_2x_2=x,\,\, \lambda_i\geq 0\, \}$$
  and $f(x)\leq F(x)$ for all $x\in[0,\oo)$. Since $F(0)=f(0)$, we can assume, without loss of generality, that $0\leq x_1<x$ and that $x_2>x$. In this situation, using \eqref{eq:concavePeetre.0.value}, we see that
 $$\lambda_1 f(x_1)+\lambda_2 f(x_2)\leq C \lambda_1 f(x) + C \lambda_2 \frac{x_2}{x} (f(x)+1)\leq
 2C f(x) +C,$$
because $\lambda_1\leq 1$ and $\lambda_2 x_2 \leq x$. Hence \eqref{eq:concavemajorant.equivalence} holds. 
\end{proof}

In the case of the largest convex minorant, which was not considered by J.~Peetre, a similar result can be obtained inspired by the previous one with a slightly different proof.

\begin{proposition}\label{prop:convexPeetre}
Let $h:(0,\oo)\to [0,\oo)$ be such that there exists $C\geq 1$  and $\b>0$ such that 
\begin{equation}\label{eq:convexPeetre}
 h(s)+C\geq \frac{1}{C} \min\left(1, \frac{t^\b}{s^\b} \right) h(t),
\end{equation}
for all $t,s\in(0,\oo)$. Then there exists $A\geq 1$ such that for every $x\in (0,\oo)$
\begin{equation}\label{eq:convexminorant.equivalence}
H(x)\leq h(x) \leq AH(x)+A,
\end{equation}
where $H$ is the largest convex minorant of $h$. 
\end{proposition}

\begin{proof}
 The largest convex minorant of $h$ can be represented by
 $$H(x)=\inf\{\lambda_1 h(x_1)+\lambda_2 h(x_2);\, \lambda_1+\lambda_2=1,\,\, \lambda_1x_1+\lambda_2x_2=x,\,\, \lambda_i\geq 0\, \}$$
 and $H(x)\leq h(x)$ for all $x\in(0,\oo)$.
 We fix $x\in (0,\oo)$, for every $x_1,x_2\in(0,\oo)$, $\lambda_1,\lambda_2\in[0,\oo)$ with $\lambda_1+\lambda_2=1$ and $\lambda_1x_1+\lambda_2x_2=x$ by \eqref{eq:convexPeetre}  we see that
 $$\lambda_1 h(x_1)+\lambda_2 h(x_2)\geq \frac{\lambda_1}{C} \min \left(1, \frac{x^\b}{x_1^\b}\right) h(x) + \frac{\lambda_2}{C} \min \left(1, \frac{x^\b}{x_2^\b}\right) h(x)-C. $$
 If $x_1=x_2=x$, then $\lambda_1 h(x_1)+\lambda_2 h(x_2)=h(x)$. Otherwise, we can assume, without loss of generality, that $x_1<x$ and that $x_2>x$. In this situation,
 $$\lambda_1 h(x_1)+\lambda_2 h(x_2)\geq \frac{h(x)}{C} \left( \lambda_1 + \lambda_2 \frac{x^\b}{x_2^\b}\right)-C=
 \frac{h(x)}{C} \left( 1 - \lambda_2 \left(1-\frac{x^\b}{x_2^\b}\right)\right)-C.$$
 Using that $\lambda_2 x_2\leq x$ and that $x_2>x$, we observe that
 $$\lambda_2 \left(1-\frac{x^\b}{x_2^\b}\right)\leq \frac{x}{x_2} \left(1-\frac{x^\b}{x_2^\b} \right)\leq \sup_{u\in[0,1]} u(1-u^\b) = \frac{1}{(1+\b)^{(\b+1)/\b}}=:K_{\b}<1.$$
 Therefore $\lambda_1 h(x_1)+\lambda_2 h(x_2)\geq h(x) (1-K_\b) C^{-1}-C$ and we conclude that \eqref{eq:convexminorant.equivalence} holds. 
\end{proof}

Thanks to the connection between the $\gamma$ index, the upper Matuszewska index and the almost monotonicity properties, it is possible to ensure that \eqref{eq:gamma.upperconjugate} and \eqref{eq:gamma.lowerconjugate} are valid under some standard assumptions.  

\begin{corollary}\label{coro:indices.upperconjugate.gama.greater.than1}
Let $\sigma$ be as above. If $\gamma(\sigma)>1$, then \eqref{eq:gamma.upperconjugate}  holds. 
\end{corollary}
\begin{proof}
If $\gamma(\sigma)>1$, then $\a(\sigma)<1$ so $\sigma(t)/t$ is almost decreasing in $[a,\oo)$ with $a>0$ such that $\sigma(t)>0$ in $[a,\oo)$. Since $\sigma(t)\leq \sigma(a)$ for all $t\in[0,a]$ and, by the monotonicity of $\sigma$, we see that \eqref{eq:concavePeetre.0.value} holds. Hence, by Proposition~\ref{prop:concavePeetre},  $\sigma$ is equivalent to its least concave majorant and by Remark~\ref{rema:om2om5omnq} it satisfies \hyperlink{om5}{$(\o_5)$}, then we conclude applying Proposition~\ref{prop:indices.upperconjugate}. 
 \end{proof}

\begin{corollary}\label{coro:indices.lowerconjugate}
Let $\sigma$ be as above. If $\gamma(\sigma)>0$, then \eqref{eq:gamma.lowerconjugate} holds. 
\end{corollary}
\begin{proof}
By Lemma~\ref{lemma.alpha.gamma}, $\a(\sigma)<\oo$, so by Theorem~\ref{th:almostmonotonerep.alpha.beta} there exists $a> 0$ and $\a>0$ such that $\sigma(t)t^{-\a}$ is almost decreasing in $[a,\oo)$. Consequently, $\sigma^{\iota}(t)t^{\a}$ is almost increasing in $(0,1/a]$. Since $\sigma^\iota$ is bounded in $[1/a,\oo)$ by $\sigma^{\iota}(1/a)$ and $\sigma^\iota$ is nonincreasing we deduce that \eqref{eq:convexPeetre} is valid and we conclude by Propositions~\ref{prop:convexPeetre} and~\ref{prop:indices.lowerconjugate}. 
 \end{proof}

\begin{remark}\label{rema.gammaoverline}
 Regarding the index $\overline{\gamma}(\sigma)$, see~\eqref{lowernewindex2}, and following similar ideas as those in Propositions~\ref{prop:indices.upperconjugate} and~\ref{prop:indices.lowerconjugate}, we can see that:
\begin{itemize}
\item[(i)] Let $\sigma$ be as above. If  $\sigma$ satisfies \hyperlink{om5}{$(\omega_{\on{5}})$} and it is equivalent to its least concave majorant, then
$$\overline{\gamma}(\sigma)=\overline{\gamma}((\sigma^{\star})^{\iota})+1.$$
Consequently $\sigma$ satisfies \hyperlink{om6}{$(\omega_6)$}, i.e. $\overline{\gamma}(\sigma)<+\infty$, if and only if $(\sigma^{\star})^{\iota}$ has \hyperlink{om6}{$(\omega_6)$}.
\item[(ii)] Let $\sigma$ be as above. If $\sigma^\iota$ is equivalent to its largest convex minorant, then
$$\overline{\gamma}(\sigma)+1=\overline{\gamma}((\sigma^{\iota})_{\star}).$$
Hence, $\sigma$ satisfies \hyperlink{om6}{$(\omega_6)$} if and only if $(\sigma^{\iota})_{\star}$ does so.
\end{itemize}
\end{remark}

\begin{remark}
The relations \eqref{eq:gamma.upperconjugate} and \eqref{eq:gamma.lowerconjugate} for the index $\ga$ can be rewritten in a different form that recalls the classical relation for conjugate indices appearing in the study of the convex conjugates. 
For this purpose, one needs to consider the notion of O-regular variation at $0$.
We say that a measurable function $h:(0,a)\to (0,\oo)$ with $a>0$ is {\it O-regularly varying at $0$}, if $h^{\iota}:(1/a,\oo)\to (0,\oo)$ is O-regularly varying. This notion has already appeared in different works dealing with O-regular variation and Orlicz spaces, see \cite{Maligranda85} . We are interested in the following quantities:
\begin{equation*}\label{matu1at0}
\alpha^0(h):=\inf\{\alpha\in\RR: x\mapsto\frac{h(x)}{x^{\alpha}}\;\text{is almost decreasing}\},
\end{equation*}
\begin{equation*}\label{matu2aot}
\beta^0(h):=\sup\{\beta\in\RR: x\mapsto\frac{h(x)}{x^{\beta}}\;\text{is almost increasing}\}.
\end{equation*}
By Theorem~\ref{th:almostmonotonerep.alpha.beta}, one might check that $\a^0(h)= -\b (h^{\iota})$ and $\b^{0} (h)= -\a (h^\iota)$. We put $\a^\oo(\sigma)=\a(\sigma)$ and $\b^\oo(\sigma)=\b(\sigma)$. Then \eqref{eq:gamma.upperconjugate} can be expressed as
 \begin{equation*}\label{eq:alpha.beta.upperconjugate}
  \frac{1}{\a^\oo(\sigma)}+\frac{1}{\b^0(\sigma^{\star})}=1,
 \end{equation*}
 and \eqref{eq:gamma.lowerconjugate} can also be written in terms of $\a^\oo((\sigma^\iota)_{\star})$ and $\b^0(\sigma^\iota)$. Moreover, under suitable assumptions, similar relations can be obtained for $\a^0 (\sigma^\star)$ and $\b^\oo(\sigma)$.
\end{remark}

\section{Weight sequences and O-regular variation}

\subsection{Weight sequences and growth indices}

In what follows, $\M=(M_p)_{p\in\N_0}$ always stands for a sequence of positive real numbers, and we always impose that $M_0=1$, where $\N_0=\{0,1,2,\dots \} = \N\cup \{0\}$.
The names of the conditions given by V.~Thilliez and, for the convenience of the reader, the corresponding descriptive
acronyms employed by the third author~\cite{Schindl16} have been used. We say that:
\begin{itemize}
\item[(i)]  $\M$ is \textit{logarithmically convex} (for short, \hypertarget{lc}{(lc)}) if for all $p\in\N$, $M_{p}^{2}\le M_{p-1}M_{p+1}$.
\item[(ii)]  $\M$ is of or has \textit{moderate growth} (briefly, \hypertarget{mg}{(mg)}) whenever there exists $A>0$ such that
$$M_{p+q}\le A^{p+q}M_{p}M_{q},\qquad p,q\in\N_0.$$%
\item[(iii)]  $\M$ satisfies the \textit{strong nonquasianalyticity condition} (for short, \hypertarget{snq}{(snq)}) if there exists $B>0$ such that
$$
\sum^\oo_{q= p}\frac{M_{q}}{(q+1)M_{q+1}}\le B\frac{M_{p}}{M_{p+1}},\qquad p\in\N_0.$$%
\end{itemize}
If $\M$ is \hyperlink{lc}{(lc)} and $\lim_{p\to\oo} (M_p)^{1/p}=\oo$ we say that $\M$ is a \textit{weight sequence}. According to V.~Thilliez~\cite{Thilliez03}, if $\M$ is \hyperlink{lc}{(lc)}, has \hyperlink{mg}{(mg)} and satisfies \hyperlink{snq}{(snq)}, we say that $\M$ is a \textit{strongly regular sequence}.\para

For a sequence $\M$ we define {\it the sequence of quotients}  $\m=(m_p)_{p\in\N_0}$ by 
$$m_p:=\frac{M_{p+1}}{M_p} \qquad p\in \N_0.$$%
We observe that for every $p\in\N$ one has $M_p= m_{p-1}m_{p-2}\cdots m_1m_0$ and $M_0=1$.
Hence the knowledge of one of the sequences
amounts to that of the other. Consequently one may work directly with $\m$ in the next subsections and
the sequences of quotients of sequences $\M$, $\L$, etc. will be denoted by lowercase letters $\bm$, $\bl$ and so on without  misunderstanding. Moreover, 
\hyperlink{snq}{(snq)} can be easily stated in terms of the sequence of quotients and, as the next lemma shows, the same holds for \hyperlink{lc}{(lc)}, \hyperlink{mg}{(mg)}, and the notion of weight sequence. Statements (i) and (ii) are classical and deduced. from a direct computation and (iii) was given by H.-J. Petzsche and D. Vogt~\cite[Lemma 5.3]{PetzscheVogt}.   

\begin{lemma}\label{lemma.basic.properties.seq}
 For every sequence $\M$ the following holds: 
\begin{enumerate}[(i)]
 \item $\M$ is logarithmically convex if and only if $\m$ is nondecreasing. 
 \item If $\M$ is logarithmically convex, $\lim_{p\to\oo} (M_p)^{1/p}= \oo$ if and only if $\lim_{p\to\oo} m_p= \oo$.  Hence $\M$ is a weight sequence if and only if $\m$ is nondecreasing and $\lim_{p\to\oo} m_p=\oo$.
 \item If $\M$ is logarithmically convex, then the following are equivalent:
 \begin{enumerate}
  \item[(iii.a)] $\M$ has \hyperlink{mg}{(mg)},
  \item[(iii.b)]  $\sup_{p\geq 1} m_p/M^{1/p}_p<\oo$,
  \item[(iii.c)] $\sup_{p\geq 1} m_{2p}/m_p<\oo$.
 \end{enumerate}
\end{enumerate}
\end{lemma}

 We say that $\M$ and $\L$ are {\itshape equivalent} and we write $\M\hypertarget{approx}{\approx}\L$ if there exists $C\geq 1$ such that for all $p\in\N_0$ $C^{-p} L_p\leq M_p\leq C^{p} L_p$. There is also an equivalence relation at the level of the sequence of quotients, we write $\m\hypertarget{simeq}{\simeq}\l$ if there exists $c\geq 1$ such that for all $p\in\N_0$, $c^{-1}\ell_p\leq m_p\leq c \ell_p$. If $\m\simeq\l$, then $\M\approx\L$ and the converse fails in general. However, 
 if $\M$ and $\L$ are \hyperlink{lc}{(lc)} and one of them has \hyperlink{mg}{(mg)}, then $\M\approx\L$ implies $\m\simeq\l$.\parb

The growth index $\gamma (\M)$ was defined and considered by V. Thilliez~\cite[Sect.\ 1.3]{Thilliez03} 
in the study of ultraholomorphic classes of functions. The original definition was given for strongly regular 
sequences and $\ga>0$, but one can consider it for any sequence $\M$ and $\ga\in\R$. We say $\M$ satisfies
property $\left(P_{\ga}\right)$  if there exists a sequence of real numbers $\l=(\ell_{p})_{p\in\N_0}$  
 such that:
\begin{enumerate}[(i)]
 \item $\m\simeq\l$, that is, there is a constant $a\ge1$ such that $a^{-1}m_{p}\le \ell_{p}\le am_{p}$, for all $p\in\N_0$,
 \item $\left((p+1)^{-\ga}\ell_{p}\right)_{p\in\N_0}$ is nondecreasing.
\end{enumerate}
If $(P_\ga)$ is satisfied, then $(P_{\ga'})$ is satisfied for $\ga'\leq\ga$.
It is natural to consider its \textit{growth index} $\ga(\M)$ defined by
$$\ga(\M):=\sup\{\ga\in\R:\, (P_{\ga})\hbox{ is fulfilled}\}.$$%
with the conventions in Remark~\ref{rema:inf.sup.empty}.\parb

For the study of the injectivity of the asymptotic Borel map, see~\cite{Sanzflat14}, the second author has defined the \textit{growth index}  $\o(\M)$ for any sequence $\M$ by
$$\o(\M):=\liminf_{p\to\oo} \frac{\log m_p }{\log p}.$$

By definition, the value of  $\ga(\M)$ and of $\o(\M)$ is stable for $\simeq$. For weight sequences with \hyperlink{mg}{(mg)}, in particular for strongly regular sequences, these values are also stable for $\approx$, thanks to the equivalence between $\approx$ and $\simeq$.  In Corollary~\ref{coro:stability.gamma.weak.equiv} and Remark~\ref{rema:stabilityOmM}, we will eventually show the stability under $\approx$ for arbitrary weight sequences. \para

We mention some interesting examples. In particular, those in (i) and (iii) appear in the applications of summability theory to the study of formal power series solutions for different kinds of equations.
\begin{itemize}
\item[(i)] The sequences $\M_{\a,\b}:=\big(p!^{\a}\prod_{m=0}^p\log^{\b}(e+m)\big)_{p\in\N_0}$, where $\a>0$ and $\b\in\R$, are strongly regular (in case $\b<0$, the first terms of the sequence have to be suitably modified in order to ensure \hyperlink{lc}{(lc)}). In case $\b=0$, we have the best known example of a strongly regular sequence, $\G_{\a}:=\M_{\a,0}=(p!^{\a})_{p\in\N_{0}}$, called the \textit{Gevrey sequence of order $\a$}. In this case, $\gamma(\M_{\a,\b})=\o(\M_{\a,\b})=\a$.
\item[(ii)] The sequence $\M_{0,\b}:=(\prod_{m=0}^p\log^{\b}(e+m))_{p\in\N_0}$, with $\b>0$, is \hyperlink{lc}{(lc)}, \hyperlink{mg}{(mg)} and $\bm$ tends to infinity, but \hyperlink{snq}{(snq)} is not satisfied and $\gamma(\M_{\a,\b})=\o(\M_{\a,\b})=0$.
\item[(iii)] For $q>1$, $\M_q:=(q^{p^2})_{p\in\N_0}$ is \hyperlink{lc}{(lc)} and \hyperlink{snq}{(snq)}, but not \hyperlink{mg}{(mg)} and $\gamma(\M_q)=\o(\M_q)=\oo$.
\end{itemize}

Most of the classical examples of strongly regular sequences satisfy that $\o(\M) = \ga(\M)$. Moreover, in the next section we will show that the values of the indices coincide for a large class of sequences, the ones whose sequence of quotients is regularly varying. However, it is possible to construct a strongly regular sequence for which the values are different, arbitrarily chosen, positive real numbers 
(see Remark~\ref{rema:SRSeqAllindicesDifferent}). 

\subsection{O-regularly varying  sequences}

In 1973, R. Bojani\'{c} and E.~Seneta~\cite{BojanicSeneta} show that, under a suitable adaptation, one may consider
regularly varying sequences satisfying similar properties to the ones of regularly varying functions. Even if all the results in this subsection, except the last one, were shown by R.~Bojani\'{c} and E.~Seneta, we refer to~\cite{BingGoldTeug89} for the proofs, as in the previous sections. This notion may be too restrictive,  we need to consider  O-regular variation for sequences, stated by S.~Aljan\v{c}i\'{c}
in 1981 and detailed by D.~Djur\v{c}i\'{c} and V.~Bo\v{z}in~\cite{DjurcicBozin} in 1997. Until the end of the  section: \parb

\centerline{\textbf{$\ba=(a_p)_{p\in\N}$ and $\bb=(b_p)_{p\in\N}$ are sequences of positive real numbers.}}\parb

The sequence $\ba$ is said to be {\it regularly varying} if
$\lim_{p\to\oo} a_{\lfloor \lambda p\rfloor} /a_p \in (0,\oo)$,
for every $\lambda\in(0,\oo)$ and {\it O-regularly varying} if
$ \limsup_{p\to\oo} a_{\lfloor \lambda p\rfloor} /a_p <\oo$,
for every $\lambda\in(0,\oo)$.\parb

Regularly and O-regularly varying sequences are embeddable as regularly and, respectively, O-regularly varying step functions.

\begin{theo}[\cite{BingGoldTeug89}, Th.\ 1.9.5, \cite{DjurcicBozin}, Th.\ 1]\label{th:embed.RVS.ORSV}
Let $\ba$ be as above and $f_{\ba} (x):=a_{\lfloor x\rfloor}$ for $x\geq 1$. Then
\begin{enumerate}[(i)]
\item  $\ba$ is regularly varying if and only if the function $f_{\ba}$ is regularly varying.
 \item  $\ba$ is O-regularly varying if and only if the function $f_{\ba}$ is O-regularly varying.
\end{enumerate}
Hence, if  $\ba$ is regularly varying, then it is O-regularly varying.
\end{theo} 

Regarding O-regular variation, from this result, we obtain the Uniform Convergence Theorem, see \cite[Th.\ 2]{DjurcicBozin} and the Representation Theorem, specially interesting in the construction of pathological examples,  as in Remark~\ref{rema:SRSeqAllindicesDifferent} and Section~\ref{counterexample}.

\begin{theo}[\cite{DjurcicBozin},\ Th.\ 3, Representation Theorem for O-regularly varying sequences]\label{th:rep.ORVS}
Let $\ba$ be an O-regularly varying sequence. Then there exist bounded sequences of real numbers $(d_p)_{p\in\N}$ and $(\xi_p)_{p\in\N}$ such that
$$a_p=\exp  \left(d_p + \sum^{p}_{j=1}\frac{\xi_j}{j}\right),\quad p\in\N.$$
Conversely, such a representation for a sequence $(a_p)_{p\in\N}$ implies that it is O-regularly varying.
\end{theo}

To the best of our knowledge, the notions of orders and Matuszewska
indices for sequences have not been considered. Only in the paper by D.~Djur\v{c}i\'{c} and V.~Bo\v{z}in~\cite{DjurcicBozin} one can find relevant information from which some of our statements concerning this topic can be inferred. In this subsection, a possible formalization of these concepts, based on Theorem~\ref{th:embed.RVS.ORSV}, is proposed, providing a simple description, analyzing their behavior under elementary sequence transformations and showing some stability properties. \para

For $\ba$ as above, we define its \textit{upper Matuszewska index $\a(\ba)$}, its \textit{lower Matuszewska index $\b(\ba)$}, its \textit{upper order $\ro(\ba)$}   and its \textit{lower order $\mu(\ba)$} by
$$\a(\ba):=\a(f_{\ba}),\quad \b(\ba):=\b(f_{\ba}),\quad \ro(\ba):=\ro(f_{\ba}), \quad \mu(\ba):=\mu(f_{\ba}), $$
where $f_{\ba}(x)=a_{\lfloor x\rfloor}$ for all $x\geq1$.

\begin{remark}\label{rema.Ch.Th.ORVSEq}
By Remark~\ref{rema:ordersROMU}, it immediately follows that $\b(\ba)\leq\mu(\ba)\leq\ro(\ba)\leq \a(\ba)$
and, using Theorems~\ref{th:Ch.ORVF} and \ref{th:embed.RVS.ORSV}, we see that $\ba$ is O-regularly varying if and only if $\b(\ba)>-\oo$ and $\a(\ba)<\oo$.\para

If the sequence $\ba$ is regularly varying of index $\o\in\R$, by Theorem~\ref{th:embed.RVS.ORSV} and Remark~\ref{rema:ordersROMU}, we deduce that 
$\b(\ba)=\mu(\ba)=\ro(\ba)=\a(\ba)=\o$.
However, the equality of the indices does not imply regular variation.
\end{remark}

A sequence $\ba$ is said to be {\itshape almost increasing,}  if there exists some $M>0$ such that $ a_p \le M a_q$ for all $p,q\in\N$ with $p\leq q$. Analogously, $\ba$ is {\itshape almost decreasing}  if there exists some $m>0$ such that $m a_q\le a_p$ for all $p,q\in\N$ with $p\leq q$.
Thanks to these definitions and using that
$\lfloor x\rfloor\leq  x \leq 2 \lfloor x\rfloor$ for all $x\geq 1$, it is possible to skip the step function $f_{\ba}$ and give a simple characterization of the indices and orders only in terms of the sequence $\ba$. 

\begin{proposition}\label{prop.nice.def.Mat.ind.ord.seq}
Let $\ba$ be as above. We have that
\begin{align*}
\a(\ba)=&\inf\{\a\in\R;  \, (a_p/p^\a)_{p\in\N}\,\, \text{is almost decreasing}\}, \quad &\ro(\ba)=\limsup_{p\to\oo}\frac{\log(a_p)}{\log(p)},\\
\b(\ba)=&\sup\{\b\in\R;  \, (a_p/p^\b)_{p\in\N}\,\, \text{is almost increasing}\}
,\quad &\mu(\ba)=\liminf_{p\to\oo}\frac{\log(a_p)}{\log(p)}.
\end{align*}
\end{proposition}

When applying ramification arguments in the classes of functions defined in terms of a given weight sequence $\M$, transformations of this sequence will appear. Using this last characterization result, the indices for the transforms and for the original sequence can be compared as indicated below.

\begin{proposition}\label{prop.roots.gevrey.multiply.Mat.ind.order}
 Let $\ba$ be a sequence  of positive numbers. For any $r\in\R$ and $s>0$, we have that
 $$\a(\ba^{s})=s \a(\ba),\quad \b(\ba^{s})=s \b(\ba),\quad \ro(\ba^{s})=s \ro(\ba), \quad \mu(\ba^{s})=s \mu(\ba), $$
 where $\ba^s:= (a^s_p )_{p\in\N}$, and we also obtain that
  $$\a(\bg_r \cdot\ba)=r +\a(\ba),\quad \b(\bg_r \cdot\ba)=r+ \b(\ba),\quad \ro(\bg_r \cdot\ba)=r+ \ro(\ba), \quad \mu(\bg_r \cdot\ba)=r +\mu(\ba), $$
 where $\bg_r:=(p^r)_{p\in\N}$ and $\bg_r \cdot\ba= ( p^r a_p)_{p\in\N}$.%
\end{proposition}

One may notice the stability of  the orders, the Matuszewska indices and the notion of O-regular variation for sequences under $\simeq$. 

\begin{lemma}\label{lemma:strong.equiv.seq.Mat.ind.order}
Let $\ba=(a_p)_{p\in\N}$ and $\bb=(b_p)_{p\in\N}$ be  as above with $\ba\simeq \bb$. Then, we see that 
$$\a(\ba)=\a(\bb),\quad \b(\ba)=\b(\bb),\quad \ro(\ba)=\ro(\bb), \quad \mu(\ba)=\mu(\bb).$$
Hence $\ba$ is O-regularly varying if and only if $\bb$ also is.
\end{lemma}

In the next subsection, these results will be applied for the sequence of quotients $\m=(m_{p-1})_{p\in\N}$ of $\M$, then the stability of those indices under $\simeq$ is a first approach. However, in the context of ultraholomorphic and ultradifferentiable classes, it is always possible to switch $\M$ for an equivalent sequence under $\approx$ and so, the appropriate question is the stability for this weaker relation. A partial but sufficient solution is given at the end of the current section (see Corollary~\ref{coro:stability.gamma.weak.equiv} and Remark~\ref{rema:stabilityOmM}). \para

\begin{remark}\label{remark:shift.seq.Mat.ind.order}
Since there is not a uniform definition of the sequence of quotients, one may alternatively consider the shifted sequence $\bs_{\bm}=(m_p)_{p\in\N}$, as some authors~\cite{BonetMeiseMelikhov07,Komatsu73,Petzsche88,PetzscheVogt,SchmetsValdivia00} have done. We see below that  both approaches are equivalent in this context and one can switch from $\m$ to $\bs_{\m}$ when needed.\para

For
$\a\geq 0$ we observe that $p^\a\leq (p+1)^\a\leq 2^\a p^\a$, so for $\ga\in\R$ we deduce that $(a_p p^{\ga})_{p\in\N}$ is almost increasing (resp. almost decreasing) if and only if $(a_{p+1} p^{\ga})_{p\in\N}$ is almost increasing (resp. almost decreasing). 
Hence, even if $\ba=(a_{p})_{p\in\N}$ and the corresponding shifted sequence $\bs_{\ba}:=(a_{p+1})_{p\in\N}$ are not equivalent for $\simeq$,
we see that
$$\a(\ba)=\a({\bs}_{\ba}),\quad \b(\ba)=\b({\bs}_{\ba}),\quad \ro(\ba)=\ro({\bs}_{\ba}), \quad \mu(\ba)=\mu({\bs}_{\ba}).$$
Consequently, by Remark~\ref{rema.Ch.Th.ORVSEq}, $\ba$ is O-regularly varying if and only if ${\bs}_{\ba}$ also is. 
\end{remark}

\subsection{Orders, Matuszewska indices and growth indices} 
In this subsection, we will see that the indices $\ga(\M)$ and $\o(\M)$ can be expressed in terms of the lower Matuszewska index $\b(\m)$ and the lower order $\mu(\m)$, respectively. First, we obtain the central connection between logarithmic convexity and O-regular variation which is the analogous version of Lemma~\ref{lemma:beta.alpha.monotone.function}.

\begin{lemma}\label{lemma:lc.ORVSeq}  Let $\M$ be a sequence of positive real numbers with sequence of quotients $\m=(m_{p-1})_{p\in\N}$. For any $\ga\in \R$, we have that
\begin{enumerate}[(i)]
\item If there exists $\bt=(t_p)_{p\in\N_0}$ nondecreasing such that $((p+1)^{-\ga} m_p)_{p\in\N_0}\simeq \bt$, then $\b(\m)\geq \ga$. Hence, if $\M$ is \hyperlink{lc}{(lc)}, then $0\leq \b(\m)\leq\mu(\m)\leq\ro(\m)\leq \a(\m)$.
\item If $\b(\m)>\ga$, then there exists $\bt=(t_p)_{p\in\N_0}$ nondecreasing such that $((p+1)^{-\ga} m_p)_{p\in\N_0}\simeq \bt$.
\end{enumerate}
\end{lemma}

According to Remark~\ref{remark:shift.seq.Mat.ind.order} and Proposition~\ref{prop.nice.def.Mat.ind.ord.seq}, the lower order $\mu (\m)$ and the growth index $\o(\M)$ coincide for any sequence $\M$. The relation between $\ga(\M)$ and Matuszewska indices can be deduced from the previous result. A weaker version of it, for strongly regular sequences and $\ga>0$ is contained in~\cite[Prop.\ 4.15]{JimenezSanz} where the connection with O-regular variation was unknown.

\begin{theo}\label{th:equality.ga.beta.mu.omega} Let $\M$ be a sequence of positive real numbers with sequence of quotients $\m=(m_{p-1})_{p\in\N}$. Then
 $$\ga(\M)=\b(\m),\qquad \o(\M)=\mu(\m).$$
\end{theo}

From this main connection, we see that the properties satisfied by $\b$ and $\mu$ appearing in  Remarks~\ref{rema.Ch.Th.ORVSEq} and~\ref{remark:shift.seq.Mat.ind.order}, Proposition~\ref{prop.roots.gevrey.multiply.Mat.ind.order}  and Lemma~\ref{lemma:lc.ORVSeq} also hold for $\gamma(\M),\o(\M)$. Some of these properties were already proved by V. Thilliez and the second author after introducing the corresponding indices. In particular, for $\M$ \hyperlink{lc}{(lc)} this means that $0\leq\ga(\M)\leq\o(\M)\leq \oo$ and we deduce that for \hyperlink{lc}{(lc)} sequences the original definition of $\ga(\M)$ given by V.~Thilliez, where the supremum is taken only for $\ga>0$, coincides with the general one considered in this paper.

\subsection{Main Theorems} Applying Theorems~\ref{th:main.th.alpha} and \ref{th:main.th.beta} to suitable step functions, we obtain similar results for the indices $\b(\m)$ and $\a(\m)$ where $\m=(m_{p-1})_{p\in\N}$ is the sequence of quotients of a weight sequence $\M$, so $\m$ is nondecreasing and tends to infinity, and the same holds for $f_{\m}$ and $f_{\bs_{\m}}$.

\begin{theorem}\label{th:main.beta.sequences}
 Let $\M$ be a weight sequence and $\b\geq 0$. The following are equivalent:  
 \begin{enumerate}[(i)]
\item  there exists $C>0$ such that $\displaystyle \sum^\oo_{k= p} \frac{ (k+1)^{\b-1}}{m_k}\leq C \frac{ (p+1)^\b}{m_p}$ for all  $p\in\N_0$,
\item there exists $\ep>0$ such that $( m_p/p^{\b+\ep} )_{p\in\N}$ is almost increasing,
\item there exists a sequence $\bh\simeq \m$ such that $((p+1)^{-\b}h_p)_{p\in\N_0}$ is nondecreasing and
 $ \displaystyle \inf_{p\geq1} \frac{h_{2p}}{h_p}> 2^{\b},$
\item $\displaystyle \lim_{k\to\oo}\liminf_{p\to\oo} \frac{m_{kp}}{k^{\b} m_p}=\oo$,
\item  there exists $k\in\N$, $k\geq 2$ such that $\displaystyle \liminf_{p\to\oo} \frac{m_{kp}}{m_p}>  k^{\b},$
\item for every $\theta \in (0,1)$ there exists $k\in\N$, $k\ge 2$,  such that $m_p\leq \theta k^{-\b} \, m_{kp}$ for every $p\in\N$,
\item  $\b(\bm)>\b$,
\item $\gamma(\M)>\b$,
\item there exists $C>0$ such that  
$\displaystyle \sum_{k=0}^{p} \frac{m_k}{(k+1)^{1+\b}}\leq C \frac{ m_p}{(p+1)^\b}$ for every $p\in\N_0$.
\end{enumerate}
\end{theorem}

\begin{proof}
(i)$\Rightarrow$(ii) By Theorem~\ref{th:main.th.beta}[(vii) $\Rightarrow$ (v)] for $\sigma_1 (t)=m_{\lfloor t\rfloor -1}$ for $t\geq 1$ and $\sigma_1(t)=m_0$ for $t\in[0,1]$, there exists $\ep>0$ such that $(m_{p}/ (p+1)^{\b+\ep})_{p\in\N_0}$ is almost increasing.  Hence, by Remark~\ref{remark:shift.seq.Mat.ind.order}, (ii) holds.\para
(ii)$\Rightarrow$(iii)  For every $p\in\N$ we define the sequence $h_p:= p^{\b+\ep} \inf_{q\geq p} ( q^{-(\b+\ep)} m_q )$
and $h_0:=h_1/2^\b$. With a direct calculation, we observe that $((p+1)^{-\b} h_p)_{p\in\N_0}$ is nondecreasing and, by (ii),  
we see that there exists $C>1$ such that $m_p/C\leq h_p \leq m_p$ for all $p\in\N$, so $\bh\simeq\m$. Finally, for all $p\in\N$ we compute
$$\frac{h_{2p}}{h_p}= 2^{\b+\ep} \frac{\inf_{q\geq 2p} ( q^{-(\b+\ep)} m_q )}{\inf_{q\geq p} ( q^{-(\b+\ep)} m_q )} \geq 2^{\b+\ep}.$$
Hence $\inf_{p\geq 1} h_{2p}/h_p > 2^{\b}$.\para

(iii)$\Rightarrow$(iv) By (iii), there exists $\ep>0$ such that $\inf_{p\geq 1} h_{2p}/h_p > 2^{\b+\ep}$. Then for all $k\in\N$, $k\geq 2$, there exists $n\in\N$ such that $2^n\leq k< 2^{n+1}$ and for every $p\in\N$ we observe that
$$\frac{h_{kp}}{k^\b h_p} \geq \frac{h_{2^n p}}{(2^{n+1})^{\b} h_p}\geq 2^{n\ep-\b}.$$
Therefore (iv) is valid for $\bh$, then it also holds for $\m$ because $\bh\simeq\m$. \para

(iv)$\Rightarrow$(v) Immediate.\para

(v)$\Rightarrow$(vi) It follows from Theorem~\ref{th:main.th.beta}[(ix) $\Rightarrow$ (viii)] for $\sigma_2(t)=m_{\lfloor t\rfloor}$ for all $t\geq 0$. \para

(vi)$\Rightarrow$(vii) For $\sigma_2(t)=f_{\bs_{\m}}(t)=m_{\lfloor t\rfloor}$ for all $t\geq 0$, Theorem~\ref{th:main.th.beta}[(viii) $\Rightarrow$ (iv)] implies $\b(\sigma)=\b(f_{\bs_{\m}})=\b(\bs_{\m})>\b$ and, by Remark~\ref{remark:shift.seq.Mat.ind.order}, $\b(\m)>\b$. \para
 
(vii)$\Leftrightarrow$(viii) By Theorem~\ref{th:equality.ga.beta.mu.omega}. \para

(vii)$\Leftrightarrow$(ix) $\Leftrightarrow$(i) By Theorem~\ref{th:main.th.beta} for $\sigma_1 (t)=f_{\m}(t)=m_{\lfloor t\rfloor -1}$ for $t\geq 1$ and $\sigma_1(t)=m_0$ for $t\in[0,1]$.

\end{proof}

The importance of Theorem~\ref{th:main.beta.sequences} lies on the relation between the listed conditions and some of the classical properties for weight sequences appearing in the literature. For $\b=0$, (i) is precisely \hyperlink{snq}{(snq)} for $\M$ and for $\b=1$ it is condition $(\ga_1)$ of H.-J. Petzsche~\cite{Petzsche88} for $\m$. The authors have employed, when studying the surjectivity of the Borel map~\cite{JimenezSanzSchindlInjectSurject}, an extension of $(\ga_1)$ introduced by J. Schmets, M. Valdivia~\cite{SchmetsValdivia00} for $r\in\N$, that can be defined for $r>0$ as follows: we say that  $\m$ satisfies \hypertarget{gar}{$(\ga_{r})$} if there exists $C>0$ such that
$$(\ga_{r})\qquad \sum^\oo_{k=p} \frac{1}{(m_k)^{1/r}}\leq \frac{C (p+1) }{(m_p)^{1/r}},  \qquad p\in\N_0.$$
In particular, the next corollary was needed in the cited work.

\begin{corollary}\label{coro:snq.ga1.gabeta}
 Let $\M$ be a weight sequence and $\b>0$. Then
 \begin{enumerate}[(i)]
  \item $\ga(\M)>0$ if and only if $\M$ satisfies \hyperlink{snq}{(snq)}. 
\item $\ga(\M)>1$ if and only if $\m$ satisfies \hyperlink{gar}{$(\ga_{1})$} .
\item $\ga(\M)>\b$ if and only if $\m$ satisfies \hyperlink{gar}{$(\ga_{\b})$} .
 \end{enumerate}
\end{corollary}

\begin{proof}
 (i) and (ii) follow directly from Theorem~\ref{th:main.beta.sequences}. (iii) First, we note that $\M^{1/\b}$ is also a weight sequence, so by Proposition~\ref{prop.roots.gevrey.multiply.Mat.ind.order} and (ii), 
  $\ga(\M)>\b$ if and only if $\ga(\M^{1/\b})>1$  if and only if $\m^{1/\b}$ satisfies \hyperlink{gar}{$(\ga_{1})$} if and only if $\m$ satisfies \hyperlink{gar}{$(\ga_{\b})$}. 
\end{proof}

The reader might have noticed that condition (v) in Theorem~\ref{th:main.beta.sequences} appears frequently in the context of weighted spaces in the case $\b=0$ as condition $(\b_3)$ in~\cite{Schindldissertation} or as condition 12.(2) in~\cite{BonetMeiseMelikhov07} and in the case $\b=1$ as $(\b_1)$ in~\cite{Petzsche88,Schindldissertation}. It is worthy to mention that, due to the previously  commented index shift nuisance, sometimes these conditions are asked to be satisfied by $\m$ and others by $\bs_{\m}$, but thanks to Remark~\ref{remark:shift.seq.Mat.ind.order}, we know that both approaches are equivalent.\parb

In the literature, see H. Komatsu~\cite{Komatsu73}, H.-J. Petzsche~\cite{Petzsche88}, J.~Bonet, R.~Meise and S.N.~Melikhov~\cite{BonetMeiseMelikhov07},
the Carleman ultradifferentiable classes are usually defined by imposing control of the derivatives as in \eqref{eq:estimation.f.factorialIN}. However, in some cases, and specially when dealing with ultraholomorphic classes (see, for example, \cite{ChaumatChollet94}, \cite{JimenezSanzSchindlInjectSurject}, \cite{Thilliez03}), 
the control as in~\eqref{eq:estimation.f.factorialOUT}  is preferred, so highlighting the defect of analyticity of the functions belonging to the class. 
Accordingly, the properties of the class are deduced from conditions on the sequence $\widehat{\M}:=\mathbb{G}_1\M=(p!M_p)_{p\in\N_0}$ or on the sequence $\M$, respectively. Even if most of the conditions can be translated from one approach to the other in the first case $\widehat{\M}$ is often supposed to be a weight sequence whereas in the second situation the basic hypothesis is satisfied by $\M$. 
This difference might be troublesome since if $\M$ is a weight sequence then $\widehat{\M}$ also is, but the opposite is not true in general. Hence one might think if Theorem~\ref{th:main.beta.sequences} is valid under weaker assumptions. In this sense, we observe the following:  

\begin{corollary}\label{coro:widehat.M.beta}
Let $\b\geq0$ and $\M$ be a sequence such that for some $r\geq 0$ the sequence $\mathbb{G}_r \M :=(p!^r M_p)_{p\in\N_0}$ is a weight sequence. Then the equivalences for $\M$ and $\b$ in Theorem~\ref{th:main.beta.sequences} hold.  
\end{corollary}

\begin{proof}
Using  Proposition~\ref{prop.roots.gevrey.multiply.Mat.ind.order} and Remark~\ref{remark:shift.seq.Mat.ind.order}, we can check that each condition in Theorem~\ref{th:main.beta.sequences}  holds for $\m$ and $\b$ if and only if this same condition is valid for $( p^r m_{p-1})_{p\in\N}$ and $\b+r$. Hence we deduce the desired equivalence for the conditions applying Theorem~\ref{th:main.beta.sequences} to $\mathbb{G}_r \M$. 
\end{proof}

By Lemma~\ref{lemma:strong.equiv.seq.Mat.ind.order}, we know that the value of the index $\b(\m)=\ga(\M)$ is stable for $\simeq$. Consequently, all the conditions listed in Theorem~\ref{th:main.beta.sequences}  are stable for $\simeq$ for weight sequences. However, the natural requirement for weight sequences is the stability for the weaker relation $\approx$. In~\cite[Th.\ 3.4]{Petzsche88}, the stability of \hyperlink{gar}{$(\ga_1)$} condition for $\approx$ is indirectly deduced and to the best of our knowledge there is no direct proof of this fact which is used  to obtain the desired stability as follows.

\begin{corollary}\label{coro:stability.gamma.weak.equiv}
 Let $\M$ and $\L$ be sequences with $\M\hyperlink{approx}{\approx}\L$. Assume that there exists $r\ge 0$ such that $\G_r \M$ and $\G_r\L$ are weight sequences. Then $\gamma(\M)=\gamma(\L)$ holds true.
Consequently, all the conditions in Theorem~\ref{th:main.beta.sequences} are stable for $\approx$ for such sequences. 
\end{corollary}

\begin{proof}
We might assume without loss of generality that $\M$ and $\L$ are weight sequences, otherwise, since $\G_r\M\approx \G_r \L$ also holds, one will show that $\ga(\G_r\M)=\ga(\G_r\L)$ and we conclude using Proposition~\ref{prop.roots.gevrey.multiply.Mat.ind.order} and Corollary~\ref{coro:widehat.M.beta}.\para

 Let $0<s<\gamma(\M)$, then $\M^{1/s}$ and $\L^{1/s}$ are also a weight sequence with $\M^{1/s}\hyperlink{approx}{\approx}\L^{1/s}$ and, by Proposition~\ref{prop.roots.gevrey.multiply.Mat.ind.order}, we obtain that $\gamma(\M^{1/s})>1$. By Corollary~\ref{coro:snq.ga1.gabeta}.(ii), applied to $\M^{1/s}$,  we see that $\m^{1/s}$ satisfies \hyperlink{gar}{$(\ga_1)$}. By~\cite[Th.\ 3.4]{Petzsche88}, we deduce that  $\l^{1/s}$ has \hyperlink{gar}{$(\ga_1)$}.
Then, reversing the arguments above for $\L$ instead of $\M$, we get $s<\gamma(L)$ and which proves $\gamma(\M)\le\gamma(\L)$. The converse inequality follows analogously.
\end{proof}

\begin{remark}
 One can go one step further and study what remains true if $\M$ is an arbitrary sequence, not necessarily a weight sequence, such that $\m \simeq \l$ with $\L$ a weight sequence. In this case, since the value of $\ga$ index and condition \hyperlink{gar}{$(\ga_1)$} are stable for $\simeq$ for any pair of sequences, applying Theorem~\ref{th:main.beta.sequences} for $\L$ and $\b=1$, we see that $\ga(\M)>1$ if and only if $\m$ has \hyperlink{gar}{$(\ga_1)$}. \para 
 Moreover, since $\gamma(\M)>1$ implies 
  that there exists $\ep>0$ such that $(m_p/p^{1+\ep})_{p\in\N}$ is almost increasing and tends to infinity, as in Theorem~\ref{th:main.beta.sequences}~[(ii) $\Rightarrow$ (iii)] one might construct a weight sequence $\L= (H_p/p!)_{p\in\N_0}$ with $\m \simeq ((p+1) \ell_p)_{p\in\N}$. Hence,  for an arbitrary sequence $\M$ we see that the following are equivalent:
\begin{itemize}
\item[(i)] $\gamma(\M)>1$, 

\item[(ii)] there exists  a weight sequence $\L$ such that  $\m \simeq ((p+1) \ell_p)_{p\in\N_0}$ and $\widehat{\l}$ satisfies \hyperlink{gar}{$(\ga_1)$},

\item[(iii)] there exists  a weight sequence $\H$ such that  $\m \simeq \bh$ and $\bh$ satisfies \hyperlink{gar}{$(\ga_1)$},

\end{itemize}
 
\end{remark}

Subsequently, we obtain an analogous result for the index $\a(\m)$ which, as we will show, is related to moderate growth condition.

\begin{theorem}\label{th:main.alpha.sequences}
  Let $\M$ be a weight sequence and $\a>0$. The following are equivalent: 
 \begin{enumerate}[(i)]
\item  there exists $C>0$ such that $\displaystyle \sum^p_{k= 0} \frac{ (k+1)^{\a-1}}{m_k}\leq C \frac{ (p+1)^\a}{m_p}$ for all  $p\in\N_0$,
\item there exists $\ep\in(0,\a)$ such that $(m_p/p^{\a-\ep})_{p\in\N}$ is almost decreasing,
\item there exists a sequence $\bh\simeq \m$ such that $((p+1)^{-\a}h_p)_{p\in\N_0}$ is nonincreasing and
 $ \displaystyle \sup_{p\geq1} \frac{h_{2p}}{h_p}< 2^{\a},$
 \item  $\displaystyle \lim_{k\to\oo}\limsup_{p\to\oo} \frac{m_{kp}}{k^\a m_p}=0,$
 \item there exists $k\in\N$, $k\ge 2$, such that $\displaystyle\limsup_{p\to\oo} \frac{m_{kp}}{m_p}<  k^\a$,
\item for every $\theta \in (0,1)$ there exists $k\in\N$, $k\ge 2$,  such that $m_{kp}\leq \theta k^{\a} \, m_{p}$ for every $p\in\N$,
\item  $\a(\bm)<\a$,
\item there exists $C>0$ such that  
$\displaystyle \sum_{k=p+1}^{\oo} \frac{m_k}{(k+1)^{1+\a}}\leq C \frac{ m_p}{(p+1)^\a}$ for every $p\in\N_0$.
\end{enumerate}
 
\end{theorem}

\begin{proof}
Analogous to the proof of Theorem~\ref{th:main.beta.sequences} using Theorem~\ref{th:main.th.alpha} instead of Theorem~\ref{th:main.th.beta}. We only note that the sequence $\h$ constructed in the proof of (ii) $\Rightarrow$ (iii) might be defined for $p\in\N$ by $h_p:= (p+1)^{\a-\ep} \sup_{q\geq p} m_q q^{\a-\ep}$ and $h_0:= h_1 /2^\a$.  
\end{proof}

Using that $\m$ is nondecreasing we see that $\sup_{p\geq 1} m_{2p}/m_p<\oo$ if and only if there exists $\a>0$ such that 
Theorem~\ref{th:main.alpha.sequences}.(v) holds. Hence, applying Lemma~\ref{lemma.basic.properties.seq}, we obtain the following corollary.

\begin{corollary}\label{coro:charact.mg} Let $\M$ be a weight sequence.
The following are equivalent:
\begin{itemize}
\item[(i)] $\M$ satisfies \hyperlink{mg}{(mg)},
\item[(ii)] there exists $\a>0$ such that $\M$ satisfies every/some of the equivalent conditions (i)--(viii) in Theorem~\ref{th:main.alpha.sequences}.
\end{itemize}
In particular, $\M$ satisfies \hyperlink{mg}{(mg)} if and only if $\a(\m)<\oo$  if and only if $\m$ is O-regularly varying.
\end{corollary}

\begin{remark}
 As it was pointed out for functions in Remark~\ref{remark:alpha.stability.open}, the set 
 $$B_{\m}:=\{\b>0; \, \text{any of the conditions in Theorem~\ref{th:main.beta.sequences} holds for}\,\, \m \, \, \text{and} \,\, \b  \}$$
 is either empty or an open subinterval of $(0,\oo)$ and, correspondingly, $\b(\m)=0$ or $\b(\m)=\sup B_\m$. Similarly, the set 
 $$A_{\m}:=\{\a>0; \, \text{any of the conditions in Theorem~\ref{th:main.alpha.sequences} holds for}\,\, \m \, \, \text{and} \,\, \a  \}$$
 is either empty or an open subinterval of $(0,\oo)$ and, correspondingly, $\a(\m)=\oo$ or $\a(\m)=\inf A_\m$.
 \end{remark}


\begin{remark}
We observe that Corollary~\ref{coro:widehat.M.beta} is also valid if we replace Theorem~\ref{th:main.beta.sequences} by Theorem~\ref{th:main.alpha.sequences}. In this situation, the stability for $\approx$ is  directly obtained using that $\simeq$ and $\approx$ are equivalent if $\M$ has moderate growth, and that $\M$ has \hyperlink{mg}{(mg)} if and only if  for some $r\geq0$ the sequence $\G_{r}\M$ has \hyperlink{mg}{(mg)}.  
\end{remark}

Finally, using the definition of the exponent of convergence of a nondecreasing sequence, see~\cite[p.\ 65]{Holland73}, as in~\cite[Th. 3.4]{Sanzflat14},  we obtain the following information for the index $\o(\M)=\mu(\m)$ of a weight sequence $\M$ which is related to  
the nonquasianalyticity condition. We recall that $\M$ is said to be {\itshape nonquasianalytic} (nq), if
$ \sum_{p=0}^{\infty} M_{p} /(p+1)M_{p+1} <+\infty$. Here the terminology may differ depending on the definition of the space, that is, if it is defined in terms of  $\widehat{\M}$, as in~\eqref{eq:estimation.f.factorialIN}, or  of $\M$, as in~\eqref{eq:estimation.f.factorialOUT}.

\begin{lemma}
 Let $\M$ be a sequence such that  $\G_r \M=(p!^r M_p)_{p\in\N_0}$ is a weight sequence for some $r\geq 0$. Then
 $$\o(\M)=\mu(\m)=\sup\{\mu>-r;   \quad  \sum^\oo_{k=0} \frac{1}{((k+1)^{r} m_k)^{1/(\mu+r)}}<\oo \}.$$
 In particular, if the hypothesis hold for some $r\in[0,1]$, then it is valid for $r=1$, that is, $\widehat{\M}$ is a weight sequence and we see that  if $\o(\M)>0$, then $\M$ is (nq) and if $\M$ is (nq), then $\o(\M)\geq 0$.  
\end{lemma}

Then the classical relations between the properties of $\M$ can be re-obtained in terms of these indices.  As it happens for functions, see Remark~\ref{rema:RVandproperties}, if $\m$ is regularly varying, as it is in many of the examples in the applications, then all the information is tied to only one value, the index  $\ro=\b(\m)=\mu(\m)=\ro(\m)=\a(\m)$ of regular variation of $\m$.

\begin{remark}
The conditions listed in the results of this section appear related to diverse problems of analysis with several different names. Due to its relevance in the context of ultradifferentiable classes it is worthy to mention that in the classical work of H.~Komatsu~\cite{Komatsu73}, whose notation has been employed by many authors, (lc) is (M.1), (mg) is (M.2) and $(\ga_1)$ is (M.3). 
\end{remark}

\section{Associated weight functions and O-Regular variation}\label{AssociatedfunctionsORV}

In this context, there is a canonical form to go from weight sequences to weight functions. Therefore, it is natural to study how the information obtained from O-regular variation affects this procedure.

\subsection{Duality between \texorpdfstring{$\m$}{m},  \texorpdfstring{$\nu_\m$}{NuM} and \texorpdfstring{$\o_\M$}{OmM}}

Let $\M$ be a sequence of positive real numbers with $M_0=1$, then the {\itshape associated function} $\omega_\M: [0,\oo)\rightarrow\RR\cup\{+\infty\}$ is defined by
\begin{equation*}\label{assofunc}
\omega_\M(t):=\sup_{p\in\N_0}\log\left(\frac{t^p}{M_p}\right)\;\;\;\text{for all}\;t>0,\hspace{30pt}\omega_\M(0):=0.
\end{equation*}
For an abstract introduction of the associated function we refer to \cite[Chap. I]{mandelbrojtbook}, see also \cite[Def. 3.1]{Komatsu73}. If $\lim_{p\to\oo} (M_p)^{1/p}=\oo$, then $\omega_\M(t)<\infty$ for any finite $t$, $\omega_\M(t)=0$ for all $t$ small enough,  $\o_{\M}$ is continuous, 
nondecreasing, $\lim_{t\to\oo} \o_\M(t)=\oo$. Moreover, if $\M$ is a weight sequence then $\o_\M$ satisfies \hyperlink{om3}{$(\omega_3)$} and \hyperlink{om4}{$(\omega_4)$}.\para

Given a weight sequence $\M$,  we may also define the \textit{counting function} for the sequence of quotients $\bm$,
$\nu_{\m}:(0,\infty)\to\N_0$ given by
\begin{equation*}
\nu_{\m}(t):=\#\{j\in\N_0:m_j\le t\}=\max\{j\in\N:m_{j-1}\le t \}.
\end{equation*}%
This function is nondecreasing and $\lim_{t\to\oo} \nu_\m(t)=\oo$. Since $\o_\M$ and $\nu_\m$ are measurable and positive, we may consider their Matuszewska indices and their upper and lower orders
 $\b(\o_{\M})$, $\b(\nu_{\m})$, $\a(\o_{\M})$, $\a(\nu_{\m})$, $\mu(\o_{\M})$, $\mu(\nu_{\m})$, $\ro(\o_{\M})$, $\ro(\nu_{\m})$, which belong to $[0,\oo]$. Hence the results in Section~\ref{WeightFandORV} connecting these indices with the classical properties are available.\para

For a weight sequence $\M$, we recover the classical relation, between $\nu_{\m}$ and the associated function $\o_{\M}$. One has that  
\begin{equation}\label{equarelaMdetnuder}
\o_{\M}(t)=\int_0^t\frac{\nu_\m(r)}{r}\,dr,\qquad t>0.
\end{equation}


The classical correspondence  between the properties of $\m$, $\nu_\m$ and $\o_\M$ (see \cite[Lemma \ 12]{BonetMeiseMelikhov07} and \cite[Lemma \ 2.2]{RainerSchindlExtension17})
suggests that one may also find  relations for their Matuszewska indices. The counting function of $\m$ can be seen as the generalized inverse of the associated function $f_{\m}$. In a recent work of D.~Djur\v{c}i\'{c}, R.~Nikoli\'{c} and A.~Torga\v{s}ev~\cite{DjurcicNikolicTorgasev10}, the connection of O-regular variation with the  generalized inverse of a positive nondecreasing unbounded function $f:[X,\oo)\to(0,\oo)$, given by 
$$f^{\leftarrow} (x) :=\inf\{y\geq X; f(y) > x \}=\sup\{y\geq X; f(y)\leq  x \} $$
for all $x\geq f(X)$, has been partially studied. Even if some information can be inferred from their proofs, there is not an explicit correspondence between the indices of $f$ and $f^{\leftarrow}$. In our situation, we will show the correspondence between the ones of the sequence of quotients and its counting function. We will start by the Matuszewska indices $\a$ and $\b$, see Remark~\ref{rema:stabilityOmM} for the information about $\ro$ and $\mu$.

\begin{proposition}\label{prop.duality.ORV.M.nuM}
Let $\M$ be a weight sequence. Then
$$\b(\m)=\frac{1}{\a(\nu_\m)},\qquad \a(\m)=\frac{1}{\b(\nu_\m)}$$
(with the typical conventions for $0$ and $\oo$). 
\end{proposition}

\begin{proof}
First we assume that $0<\ga<\b(\m)$, so, by Proposition~\ref{prop.nice.def.Mat.ind.ord.seq}, $(p^{-\ga} m_p)_{p\in\N}$  is almost increasing, then $(p^{-1} m_p{}^{1/\ga})_{p\in\N}$  is almost increasing with constant $D\geq 1$. For every $t\geq s\geq m_0$, there exist $p,q\in\N_0$ such that $q\geq p$,  $s\in[ m_{p}, m_{p+1})$ and  $t\in[m_q,m_{q+1})$. If $q=p$, we see that
$$\frac{\nu_\m(s)}{s^{1/\ga}} =  \frac{\nu_{\m} (t) }{s^{1/\ga}} \geq \frac{\nu_{\m} (t) }{t^{1/\ga}}, $$
and if $q\geq p+1\geq 1$, we get
$$\frac{\nu_\m(s)}{s^{1/\ga}} \geq \frac{p+1}{(m_{p+1})^{1/\ga}} \geq  \frac{q}{D(m_{q})^{1/\ga}}   \geq   \frac{q}{q+1}\frac{q+1}{ D t^{1/\ga}}\geq  \frac{\nu_{\m} (t) }{2 D t^{1/\ga}}, $$
that is, $\nu_{\m}(t)/t^{1/\ga}$ is almost decreasing, then $1/\ga\geq\a(\nu_\m)$ and $1/\b(\m)\geq\a(\nu_\m)$. \para 
Similarly, if $\ga>\a(\m)$, then $1/\ga\leq\b(\nu_\m)$ and $1/\a(\m)\leq\b(\nu_\m)$.  
\para
Reciprocally, if $\ga>\a(\nu_\m)$,  there exists $\ep>0$ such that $\ga-\ep>\a(\nu_\m)$, so $\ga-\ep>0$, since $\a(\nu_\m)\geq 0$. Hence,
$\nu_{\m}(t)/t^{\ga-\ep}$ is almost decreasing which implies that there exists $d\in(0,1)$ such that for every $\lambda\geq 1$ and all $t\geq m_0$  we get 
\begin{equation*} 
\nu_{\m} (t) \geq  d \nu_{\m} (\lambda t)/\lambda^{\ga-\ep}.
\end{equation*}
We fix $Q\in\N$, large enough, such that $Q^{(\ep/2)/(\ga-\ep/2)} d \geq 1$ and taking $\lambda=Q^{1/(\ga-\ep/2)}$ we see that 
\begin{equation} \label{eq.nu.almost.decreasing}
\nu_{\m} (t) Q \geq   \nu_{\m} (Q^{1/(\ga-\ep/2)} t),   \qquad   t\geq m_0.
\end{equation}
Using~\eqref{eq.nu.almost.decreasing}, for $p\in\N$, we observe that
\begin{align*}
m_p =& \sup\{t\geq m_0; \, \nu_\m(t)\leq p  \} \leq  \sup\{t\geq m_0; \, \nu_\m(Q^{1/(\ga-\ep/2)}t)\leq Q p  \}\\ 
=& Q^{-1/(\ga-\ep/2)}  \sup\{s\geq Q^{1/(\ga-\ep/2)} m_0; \, \nu_\m(s)\leq Q p  \}\leq \frac{m_{Qp}}{Q^{1/(\ga-\ep/2)}}.
\end{align*}
Hence we have shown that there exist $Q\in\N$, $Q\geq 2$ and $\de>0$ such that
$$\liminf_{p\to\oo} \frac{m_{Qp}}{Q^{1/\ga} m_{p}}\geq Q^\de>1.$$
By Theorem~\ref{th:main.beta.sequences}, we obtain that $1/\ga\leq \b(\m)$ and $1/\a(\nu_\m)\leq \b(\m)$. 
Analogously, if $0<\ga<\b(\nu_\m)$, using Theorem~\ref{th:main.beta.sequences}, we get that $1/\ga\geq \a(\m)$ and $1/\b(\nu_\m)\geq \a(\m)$.
\end{proof}

Applying Corollaries~\ref{coro:om1}, \ref{coro:om6}, \ref{coro:snq.ga1.gabeta} and~\ref{coro:charact.mg},  we recover the classical equivalence for the growth conditions of $\M$ and $\nu_\m$:

\begin{corollary}
 Let $\M$ be a weight sequence. Then
\begin{enumerate}[(i)]
 \item $\M$ has \hyperlink{mg}{(mg)} if and only if $\nu_\m$ satisfies \hyperlink{om6}{$(\omega_{\on{6}})$} .
 \item $\M$ satisfies \hyperlink{snq}{(snq)} if and only if $\nu_\m$ satisfies \hyperlink{om1}{$(\omega_{\on{1}})$}.
\end{enumerate}
\end{corollary}

Now, we want to connect the O-regular variation character of $\nu_\m$ and of $\o_\M$, which can be done using~\eqref{equarelaMdetnuder} and the next result that compares the O-regular variation of a function and its derivative.

\begin{theo}[\cite{BingGoldTeug89}, Th. \ 2.6.1, Coro. \ 2.6.2]\label{th.equiv.f.integral.ORV} Let $f:[X,\oo)\to(0,\oo)$ be a locally integrable function.  We define $F(x):=\int^{x}_{X} f(t)/t dt$. Then,
\begin{enumerate}[(i)]
 \item If $\a(f)<\oo$, then $\limsup_{x\to\oo} f(x)/F(x)<\oo$.
 \item If $\b(f)>0$, then $\liminf_{x\to\oo} f(x)/F(x)>0$.
 \item We have that $\a(F)\leq \limsup_{x\to\oo} f(x)/F(x)$.
 \item We have that $\b(F)\geq \liminf_{x\to\oo} f(x)/F(x)$.
\end{enumerate}
Moreover, we get that
$$0<\liminf_{x\to\oo} \frac{f(x)}{F(x)}\leq \limsup_{x\to\oo}\frac{f(x)}{F(x)} <\oo$$
if and only if $\a(f)<\oo$ and $\b(f)>0$. In this case, $\a(F)=\a(f)$ and $\b(F)=\b(f)$.
\end{theo}

Since $\nu_{\m}:[m_0,\oo)\to(0,\oo)$ is a locally integrable function, an easy consequence of \eqref{equarelaMdetnuder} and  Theorem~\ref{th.equiv.f.integral.ORV} is the following:

\begin{theo}\label{th:nuOmIndexRelation}
 Let $\M$ be a weight sequence. Then
 \begin{enumerate}[(i)]
  \item $\a(\nu_\m)\geq \a(\o_\M)$.
  \item $\b(\nu_\m)=\b(\o_\M)$. 
 \end{enumerate}
\end{theo}

\begin{proof}
 (i) If $\a(\nu_\m)=\oo$, the result is trivial. If $\a(\nu_\m)<\oo$, we fix $\a>\a(\nu_\m)$, then $t\mapsto \nu_\m(t)/t^\a$ is almost decreasing. By  \eqref{equarelaMdetnuder}, for every $y\geq x\geq m_0$ we see that
 $$\frac{\o_{\M}(y)}{y^\a}\leq \frac{1}{y^\a} \left( \int^x_{m_0} \nu_\m(u) \frac{d u}{u} +
 \int^y_{x} \nu_\m(u)\frac{d u}{u}  \right) \leq \frac{\o_\M (x)}{x^\a} + \frac{C}{\a} \frac{\nu_\m(x)}{x^\a}.$$
 By Theorem~\ref{th.equiv.f.integral.ORV}.(i), there exists $D>0$ such that $\nu_\m(t)\leq D \o_\M (t)$ for $t$ large enough, so  $\a\geq\a(\o_\M)$. \para \noindent

 (ii) First we will show that $\b(\nu_\m)\leq \b(\o_\M)$. If $\b(\nu_\m)= 0$, the inequality holds. Assume that $\b(\nu_\m)>\b>0$, by Theorem~\ref{th.equiv.f.integral.ORV}.(ii) there exists $D,t_0>0$ such that
 $\nu_\m(t)\geq D \, \o_\M(t)$ for $t\geq t_0$ and, by Theorem~\ref{th:main.th.beta}, $\liminf_{t\to\oo} \nu_\m(\lambda t)\lambda^{-\b} (\nu_\m(t))^{-1}>  e^\b/D$ for all $\lambda\geq \lambda_0$. 
 From  \eqref{equarelaMdetnuder} and the monotonicity of $\nu_\m$ for every $t>0$ we get
\begin{equation}\label{eq:OmNuineq}
  \omega_\M(et)=\int_{0}^{et}\nu_\m(u)\frac{du}{u}\geq\int_{t}^{et}\nu_\m(u)\frac{du}{u}\ge \nu_\m(t).
\end{equation}
Hence
$$ \liminf_{t\to\oo} \frac{\o_{\M} (\lambda e t)}{ (e\lambda)^\b \o_\M(t)}\geq \liminf_{t\to\oo} \frac{D \nu_\m(\lambda t)}{ e^\b \lambda^\b \nu_\m(t)} > 1, $$
and, by  Theorem~\ref{th:main.th.beta}, we conclude that $\b<\b(\o_\M)$ so $\b(\nu_\m)\leq\b(\o_\M)$.\para

Secondly, we will see that  $\b(\nu_\m)\geq \b(\o_\M)$. If $\b(\o_\M)=0$, the result is trivial. If $\b(\o_\M)>0$, by Corollary~\ref{coro:om6}, $\omega_\M$ has \hyperlink{om6}{$(\omega_6)$}, and by \cite[Prop. 3.6]{Komatsu73} this implies that $\M$ has \hyperlink{mg}{(mg)}. Then, by Lemma~\ref{lemma.basic.properties.seq}, this implies that there exists $A>1$ such that $\sup_{p\in\N} m_p/M^{1/p}_p\leq A <\oo$.
Hence, for every $t\geq m_0$, there exists $p\in\N_0$ such that $t\in[m_p,m_{p+1})$ and we have that
$$\o_{\M}(t)\leq \o_{\M} (m_{p+1}) = (p+1) \log\left( \frac{m_{p+1}}{M^{1/(p+1)}_{p+1}}\right) \leq 
 (p+1) \log A= \nu_{\m} (t)\log A.$$
Assume that $\b(\o_\M)>\b>0$,  by Theorem~\ref{th:main.th.beta}, $\liminf_{t\to\oo} \o_\M(\lambda e t) (\lambda )^{-\b} (\o_\M(et))^{-1}>e^\b \log A$ for all $\lambda\geq \lambda_0$, so
$$ \liminf_{t\to\oo} \frac{\nu_{\m} (\lambda e t)}{ (\lambda e)^\b \nu_\m(t)}\geq \liminf_{t\to\oo} \frac{1}{e^\b \log A}\frac{\o_\M(\lambda e t)}{ \lambda^\b \o_\M(et)} > 1, $$
and, by  Theorem~\ref{th:main.th.beta}, we conclude that $\b<\b(\nu_\m)$ then $\b(\nu_\m)\geq\b(\o_\M)$. 
\end{proof}

\begin{corollary}\label{coro:nuOmEquiv}
Let $\M$ be a weight sequence. Then the following are equivalent:
\begin{enumerate}[(i)]
 \item $\displaystyle 0<\liminf_{t\to\oo} \frac{\nu_{\m}(t)}{\o_{\M}(t)}\leq \limsup_{t\to\oo}\frac{\nu_{\m}(t)}{\o_{\M}(t)} <\oo$,
 \item $\b(\nu_{\m})>0$ and $\a(\nu_{\m})<\oo$,
 \item $\b(\o_{\M})>0$ and $\a(\o_{\M})<\oo$.
\end{enumerate}
In this case, $\b(\o_{\M})=\b(\nu_{\m})$ and $\a(\o_{\M})=\a(\nu_{\m})$.
\end{corollary}

\begin{proof}
 (i) $\Leftrightarrow$ (ii) Immediate from Theorem~\ref{th.equiv.f.integral.ORV}. \para\noindent
 (i) $\Rightarrow$ (iii)   Again by Theorem~\ref{th.equiv.f.integral.ORV}, we have that $\b(\o_{\M})=\b(\nu_{\m})>0$ and $\a(\o_{\M})=\a(\nu_{\m})<\oo$.\para\noindent
 (iii) $\Rightarrow$ (i) By Theorem~\ref{th:nuOmIndexRelation}, we see that $\b(\nu_\m)=\b(\o_\M)>0$ then, by Theorem~\ref{th.equiv.f.integral.ORV}.(ii), we see that $ \liminf_{t\to\oo} \nu_{\m}(t)/\o_{\M}(t)>0$.
 From  \eqref{equarelaMdetnuder} and the monotonicity of $\nu_\m$ we obtain~\eqref{eq:OmNuineq}. 
 Since $\a(\o_\M)<\oo$, by Theorem~\ref{th:main.th.alpha}, we get
 $$ \limsup_{t\to\oo}\frac{\nu_{\m}(t)}{\o_{\M}(t)} \leq  \limsup_{t\to\oo}\frac{\omega_\M(et)}{\o_{\M}(t)} <\oo.$$
\end{proof}

Finally, we want to show the connection between the indices of $\o_\M$ and the ones of $\M$.

\begin{corollary}\label{coro:IndMnuMOmM}
 Let $\M$ be a weight sequence. Then
 \begin{enumerate}[(i)]
  \item $\ga(\M)=\b(\m)=1/\a(\nu_\m)\leq 1/\a(\o_\M)=\ga(\o_\M)$.
  \item  $\a(\m)=1/\b(\nu_\m)=1/\b(\o_\M)$.
  \item If $\M$ has in addition \hyperlink{mg}{(mg)}, then $\ga(\M)=\ga(\o_\M)$.
 \end{enumerate}
\end{corollary}

\begin{proof}
 (i) and (ii) follow from Lemma~\ref{lemma.alpha.gamma}, Proposition~\ref{prop.duality.ORV.M.nuM}, Theorems~\ref{th:equality.ga.beta.mu.omega} and~\ref{th:nuOmIndexRelation}.\para\noindent
 
 (iii) If $\a(\o_\M)=\oo$, then by Theorem~\ref{th:nuOmIndexRelation} $\a(\nu_\m)=\oo$ and $\ga(\M)=\ga(\nu_\m)=\ga(\o_\M)=0$. Assume that $\a(\o_\M)<\oo$. Since $\M$ has \hyperlink{mg}{(mg)}, by Corollary~\ref{coro:charact.mg}, $\a(\m)<\oo$ so, by (ii) we also have that $\b(\o_\M)>0$. Hence we can apply Corollary~\ref{coro:nuOmEquiv} and deduce that $\a(\o_\M)=\a(\nu_\m)=1/\b(\m)$, as desired.
\end{proof}

We close this section with several observations.

\begin{remark}\label{rema:stabilityOmM}
In the book of A. A. Goldberg and I. V. Ostrovskii~\cite[Ch.\ 2, Th.\ 1.1]{valuedistribution}, it is shown that 
$\o_\M$ and $\nu_\m$ belong to the same growth category and from their computations it can be deduced  that
$\ro(\o_\M)=\ro(\nu_\m)$ and $\mu(\o_\M)=\mu(\nu_\m)$.\para

In~\cite[Th.\ 3.5]{JimenezSanzSchindlLCSeqandPO}, the authors have shown that for any weight sequence $\M$ the following equality holds  
$$\ro(\o_\M)=\ro(\nu_\m)=\frac{1}{\mu(\m)}=\frac{1}{\o(\M)}.$$
With a more elaborated argument, it is also possible to show that $\mu(\o_\M)=\mu(\nu_\m)=1/\ro(\m)$.\para

As a consequence, we observe that if $\M$ and $\L$ are weight sequences with $\M\approx\L$ then there exists 
$A\geq 1$ such that
$$\o_\L (A^{-1} t)\leq \o_{\M} (t)\leq \o_\L (At), \qquad t>0,$$
and we can show that $\ro(\o_{\M})=\ro(\o_{\L})$ so $\o(\M)=\o(\L)$.
Together with Corollary~\ref{coro:stability.gamma.weak.equiv}, this means that one can extend Lemma~\ref{lemma:strong.equiv.seq.Mat.ind.order}, that is, we have stability for $\approx$, for the two relevant indices $\ga(\M)$ and $\o(\M)$ in the study of the surjectivity and injectivity of the Borel map in ultraholomorphic classes in sectors, see~\cite{JimenezSanzSchindlInjectSurject}.
\end{remark}

\begin{remark}
 Corollary~\ref{coro:IndMnuMOmM} allows us to recover the following  well-known relations between the conditions of $\M$ and $\o_\M$ and explains why there is not a complete symmetry. For any weight sequence $\M$ we get
 \begin{enumerate}[(i)]
  \item $\M$ has \hyperlink{mg}{(mg)} if and only if $\omega_M$ has \hyperlink{om6}{$(\omega_6)$}. 
  \item If $\m$ satisfies \hyperlink{gar}{$(\gamma_1)$}, then $\omega_\M$ has \hyperlink{omsnq}{$(\omega_{\on{snq}})$}. 
  \item If $\M$ satisfies \hyperlink{snq}{(snq)}, then $\omega_\M$ has \hyperlink{om1}{$(\omega_1)$}. 
 \end{enumerate}
Moreover, if $\M$ has \hyperlink{mg}{(mg)}, then the implications in (ii) and (iii) can be reversed.
\end{remark}

\begin{remark}
  The example constructed in Section \ref{counterexample} provides a weight sequence $\M$ such that $\gamma(\M)=0$ and $\gamma(\omega_\M)=\infty$, which shows that the assumption \hyperlink{omsnq}{$(\omega_{\on{snq}})$} for $\omega_\M$ is really weaker than condition \hyperlink{gar}{$(\gamma_1)$} for $\m$ underlining the importance of the difference between $\gamma(\M)$ and $\gamma(\omega_\M)$ in general. 
\end{remark}

\subsection{A formula relationg the indices of \texorpdfstring{$\o_{\widehat{\M}}$}{OmHATM}  and \texorpdfstring{$\o_{\M}$}{OmM}}\label{subsect.OmM.OmhatM}

As commented before Corollary~\ref{coro:widehat.M.beta}, sometimes the ultradifferentiable classes are defined replacing the sequence $\M$ by $\widehat{\M}:=(p!M_p)_{p\in\N_0}$, see also~\eqref{eq:estimation.f.factorialIN} and~\eqref{eq:estimation.f.factorialOUT}, and the requirement of $\M$ being a weight sequence is substituted by the weaker condition  of $\widehat{\M}$ being a weight sequence. The relation between the indices of both sequences is stated in Proposition~\ref{prop.roots.gevrey.multiply.Mat.ind.order}. However, the connection of the indices of $\o_\M$ and $\o_{\widehat{\M}}$, which is motivated by  the study of the surjectivity of the Borel map, see~\cite{JimenezSanzSchindlExtensionReal, JimenezSanzSchindlLaplace}, does not follow directly from the theory of O-regular variation. The connection of the index $\ga$ of both functions is based on
the study of the Legendre conjugates carried out in subsection~\ref{subsect.LegendreconjugateIndex}. 
For this purpose, we recall the next property of $\omega^{\star}$,  for its proof we refer to \cite[Lemma 3.1.(ii) and (3.5)]{JimenezSanzSchindlLaplace}.

\begin{lemma}\label{lemma:OmMOmhatMupperconjugate}
Let $\M=(M_p)_{p\in N_0}$ be any sequence of positive real numbers with $M_0=1$ and $\lim_{p\rightarrow\infty}(M_p)^{1/p}=\infty$. Then for all $s>0$ 
\begin{equation*}
(\omega_{\widehat{\M}}^{\star}(s))^\iota\le
\omega_{\M}\left(s\right)\le
(\omega_{\widehat{\M}}^{\star}\left(se\right))^\iota .
\end{equation*}
\end{lemma}

Using the previous Lemma we can show:

\begin{corollary}\label{coro:OmhatMOmMgamma}
Let $\M=(M_p)_{p\in N_0}$ be any sequence of positive real numbers with $M_0=1$ such that $\widehat{\M}=(p!M_p)_{p\in\N_0}$ is a weight sequence and $\lim_{p\to\oo} (M_p)^{1/p}=\oo$. 
Assume that $\gamma(\omega_{\widehat{\M}})>1$. Then we obtain
$$\gamma(\omega_{\widehat{\M}})=\gamma(\omega_\M)+1.$$
\end{corollary}

\begin{proof}
By Corollary~\ref{coro:indices.upperconjugate.gama.greater.than1} applied to $\sigma=\o_{\widehat{\M}}$, we see that $\ga(\o_{\widehat{\M}})=\ga((\o^{\star}_{\widehat{\M}})^\iota)+1$. This means that
$\ga((\o^{\star}_{\widehat{\M}})^\iota)>0$, so $(\o^{\star}_{\widehat{\M}})^\iota$ satisfies $(\omega_1)$.\para

By Lemma~\ref{lemma:OmMOmhatMupperconjugate} this implies that $(\o^{\star}_{\widehat{\M}})^\iota\sim \o_\M$ and, consequently, since both functions are nondecreasing and tend to $\oo$ at $\oo$,  $\gamma(\omega_\M)=\gamma((\o^{\star}_{\widehat{\M}})^\iota)$.
\end{proof}

\subsection{Strongly regular sequences}\label{subsect:SRS}

The ultradifferentiable spaces defined in terms of a strongly regular sequence  are known to
be well-behaved with respect to extension and division properties, see~\cite{Thilliez03} and the references therein. Moreover the classical result of J.~Bonet, R.~Meise and S.N.~Melikhov~\cite[Th.\ 14]{BonetMeiseMelikhov07}, stated in a more general framework in~\cite[Sect.\ 6]{Schindl16} by the third author, can be translated into the following form: the ultradifferentiable space defined, as in \eqref{eq:estimation.f.factorialIN}, by a weight sequence $\widehat{\M}$ satisfying $\b(\widehat{\m})>0$  can be defined in terms of the associated function $\o_{\widehat{\M}}$ if and only if $\a(\widehat{\m})<\oo$. Hence the spaces coincide whenever $\widehat{\M}$ is strongly regular.\para 

With the information in Remark~\ref{rema.Ch.Th.ORVSEq}, Corollaries~\ref{coro:snq.ga1.gabeta},~\ref{coro:charact.mg} and~\ref{coro:nuOmEquiv} we obtain the next characterization of strongly regular sequences which helps us to understand their nature.

\begin{corollary}\label{Coro.SRS.Charact}
 Let  $\M$ be a weight sequence. The following are equivalent:
\begin{itemize}
\item[(i)] $\M$ is strongly regular,
\item[(ii)] There exists  $k\in\N$, $k\ge 2$, such that 
$$ 1<\liminf_{p\to\oo} \frac{m_{kp}}{m_p} \leq \limsup_{p\to\oo} \frac{m_{kp}}{m_p}<\oo,$$
\item[(iii)] $\a(\m)<\oo$ and $\b(\m)>0$,
\item[(iv)] $\m$ is O-regularly varying and $\b(\m)>0$,
\item[(v)] $\displaystyle 0<\liminf_{t\to\oo} \frac{\nu_{\m}(t)}{\o_{\M}(t)}\leq \limsup_{t\to\oo}\frac{\nu_{\m}(t)}{\o_{\M}(t)} <\oo,$ 
 \item[(vi)] $\a(\nu_{\m})<\oo$ and $\b(\nu_{\m})>0$,
 \item[(vii)] $\a(\o_{\M})<\oo$ and $\b(\o_{\M})>0$.
\end{itemize}
 In this case, we also get that $\a(\o_{\M})=\a(\nu_{\m})=1/\b(\m)$, $\b(\o_{\M})=\b(\nu_{\m})=1/\a(\m)$.
\end{corollary}

\begin{remark}\label{rema:SRSeqAllindicesDifferent}
From the last Corollary we see that, if $\M$ is strongly regular, then $\ga(\M),\o(\M) \in (0,\oo)$. Most of the classical examples of strongly regular sequences satisfy that $\o(\M) = \ga(\M)$. However, 
in~\cite[Example 4.18]{JimenezSanzSchindlLCSeqandPO} we have constructed a strongly regular sequence with $\ga(\M)=2$ and $\o(\M)=5/2$. This construction can be extended: given four mutually distinct positive values $0<\b<\mu<\ro<\a<\oo$, for all $a>b>1$ and $n\in\N_0$ we define
\begin{equation*} 
 \xi(t):= \left\lbrace
\begin{array}{l}
\a, \quad \text{for} \quad t\in[2^{a^n}, 2^{b a^{n}}), \\
\b, \quad \text{for} \quad t\in[2^{b a^n}, 2^{a^{n+1}}  ), 
\end{array}
  \right.
\end{equation*}
and $\xi(t)=1$ for all $t\in[1,2)$,  see~\cite[Sect. 2.2.5]{JimenezPhD} where other pathological examples have been constructed. We can consider the function 
$$\o(x)=\exp\left(\int^{x}_{1} \xi(u) \frac{du}{u}\right),\qquad t>1,$$
that is a nondecreasing function. A straightforward but tedious computation leads to
$$\mu(\o)=\frac{(b-1)\a+(a-b)\b}{a-1},\qquad \ro(\o)=\frac{a(b-1)\a+(a-b)\b}{b(a-1)},$$
then taking
$$b:=\frac{\a-\mu}{\a-\ro}, \qquad a:=b\,\,\frac{\ro-\b}{\mu-\b}= \frac{\a-\mu}{\a-\ro} \,\frac{\ro-\b}{\mu-\b} $$ 
we obtain 
 $$\b(\o)=\b,\quad\mu(\o)=\mu,\quad\ro(\o)=\ro,\quad \a(\o)=\a.$$
Finally, the sequence $\M$ defined in terms of its sequence of quotients by $m_p:=\o(p)$ for $p$ large enough, is strongly regular and 
$$\b(\m)=\b,\quad\mu(\bm)=\mu,\quad\ro(\bm)=\ro,\quad\a(\bm)=\a.$$ \para
\end{remark}

\subsection{Proximate orders and weight functions}
We recall now the notion of so-called proximate orders, see \cite{valuedistribution}, \cite{valiron} resp. \cite[Definitions 4.1, 4.2]{Sanzflat14} and \cite[Definition 2.1]{JimenezSanzSchindlLCSeqandPO} and the references therein.

\begin{definition}\label{proximateorder}
A function $\varrho:(c,+\infty)\rightarrow[0,+\infty)$ for some $c\ge 0$ is called a {\itshape proximate order} if the following conditions are satisfied:
\begin{itemize}
\item[$(A)$] $\varrho$ is continuous and piecewise continuously differentiable in $(c,+\infty)$,
\item[$(B)$] $\varrho(t)\ge 0$ for all $t>c$,
\item[$(C)$] $\lim_{t\rightarrow\infty}\varrho(t)=\rho<+\infty$,
\item[$(D)$] $\lim_{t\rightarrow\infty}t\varrho'(t)\log(t)=0$.
\end{itemize}
If the value $\rho$ in $(C)$ is positive, then $\varrho$ is called a {\itshape nonzero proximate order}, if $\rho=0$, then $\varrho$ is called a {\itshape zero proximate order}.

\end{definition}

Assuming that the function $d_\M(t)=\log(\o_\M(t))/\log t$ is a nonzero proximate order, or if it is close to one in the sense of \cite[Def. 4.1]{JimenezSanzSchindlLCSeqandPO}, the summability theory developed by A. Lastra, S. Malek and the second author, see~\cite{Sanzsummability}, is available. One may naturally try to check if this theory is also available for spaces of functions defined in terms of a weight function
and if this approach provides any new insight. 
We point out the classical connection between proximate orders and regular variation.

\begin{lemma}[\cite{Levin80}, Sect.\ I.12, p.32]\label{lemma:PO.to.RV}
Let $\varrho$ be a proximate order with $\lim_{t\to\oo} \varrho(t)=\ro$. Then, the function $V(t)=t^{\varrho(t)}\in R_\ro$.
\end{lemma}

If there exists a proximate order $\varrho$ and positive constants $A,B$ such that
\begin{equation*}\label{eq:AdmitPODef}
 A\leq \frac{\sigma(t)}{t^{\varrho(t)}} \leq B \qquad\text{for $t$ large enough},
\end{equation*}
we say that $\sigma$ {\it admits  $\varrho$ as a proximate order}. As a consequence of Remark~\ref{rema.properties.Matuszewska.indices} and Lemma~\ref{lemma:PO.to.RV} we obtain the following corollary. 

\begin{corollary}\label{coro:weightfunction.proximateorders}
 Let $\sigma:[0,\oo)\to[0,\oo)$, nondecreasing with $\lim_{t\rightarrow\infty}\sigma(t)=\infty$ admitting $\varrho$ as a proximate order with $\lim_{t\rightarrow\infty}\varrho(t)=\ro$. Then $\b(\sigma)=\mu(\sigma)=\ro(\sigma)=\a(\sigma)=\ro$.
 Consequently, if $\sigma$ admits a nonzero proximate order, it satisfies \hyperlink{om1}{$(\omega_1)$} and \hyperlink{om6}{$(\omega_6)$}.
\end{corollary}

In \cite{RainerSchindlcomposition} and \cite{Schindldissertation}, with each normalized weight function $\o$ satisfying 
\hyperlink{om3}{$(\omega_3)$} a weight matrix $\Omega:=\{\mathbb{W}^\ell=(W^\ell_j)_{j\in\N_0}: l>0\}$ has been associated, defined by
$$W^\ell_j:=\exp\left(\frac{1}{\ell}\varphi^{*}_{\omega}(\ell j)\right),$$
where $\varphi^{*}_{\o}(x):=\sup\{x y-\o(e^y): y\ge 0\}.$ We summarize some facts which are shown in \cite[Section~5]{RainerSchindlcomposition}:
for all $l>0$ we have $\mathbb{W}^l$ is a weight sequence and $\omega$ has \hyperlink{om6}{$(\omega_6)$} if and only if some/each $W^\ell$ satisfies \hyperlink{mg}{(mg)}.
In case \hyperlink{om6}{$(\omega_6)$}  is satisfied, $\Omega$ is constant, i.e. $\mathbb{W}^x\hyperlink{approx}{\approx}\mathbb{W}^y$ for all $x,y>0$ by recalling \cite[Lemma 5.9]{RainerSchindlcomposition}, so the associated ultradifferentiable resp. ultraholomorphic class defined by the weight $\omega$ can already be represented by a single sequence $\mathbb{W}^x$.\para

In the same way as Carleman ultradifferentiable or ultraholomorphic classes may be defined by imposing control of the derivatives by a sequence $\widehat{\M}=(p!M_p)_{p\in\N_0}$, as in \eqref{eq:estimation.f.factorialIN}, or by a sequence $\M$, as in \eqref{eq:estimation.f.factorialOUT}, and so the properties of the class are deduced from conditions on the sequence $\widehat{\M}$, respectively $\M$, one could be tempted to take two corresponding different approaches for introducing ultradifferentiable or ultraholomorphic classes with respect to a Braun-Meise-Taylor weight function or with respect to the associated weight matrix: to work with a given weight function $\omega$ playing the role of $\widehat{\M}$, or consider instead the weight $(\omega^*)^{\iota}$, playing (as explained above, see Proposition~\ref{prop:indices.upperconjugate} or Corollary~\ref{coro:OmhatMOmMgamma}) the previous role of $\M$. However, as far as the interest of proximate orders in this respect is concerned, no difference appears in both settings since, if any of both weights admits a nonzero proximate order it will have \hyperlink{om6}{$(\omega_6)$} by Corollary~\ref{coro:weightfunction.proximateorders}, the same will be true for the other weight function according to Remark~\ref{rema.gammaoverline}, and the corresponding associated weight matrices will be constant, so that the classes defined by any of these procedures can be defined by a single weight sequence and no new, richer structure is obtained in any case.

\section{A (counter)-example comparing \texorpdfstring{$\gamma(\M)$}{GammaM} and \texorpdfstring{$\gamma(\omega_\M)$}{GammaOmM}}\label{counterexample}

The objective of the example constructed in this section is to show that the inequality in Corollary~\ref{coro:IndMnuMOmM}.(i) can be strict. 
The example is suitable for both approaches, commented before Corollary~\ref{coro:widehat.M.beta} and at the beginning of Subsection~\ref{subsect.OmM.OmhatM}, that is, for classes defined in terms of $\M$ or of $\widehat{\M}=(p!M_p)_{p\in\N_0}$. According to Corollary~\ref{coro:IndMnuMOmM}.(iii), the sequences $\M$ and $\widehat{\M}$ can not satisfy \hyperlink{mg}{(mg)} , but even if \hyperlink{mg}{(mg)} is violated  
$\ga(\M)$ and $\ga(\o_\M)$ might be equal. \para

We construct a (counter)-example of a weight sequence $\M$, satisfying the following properties:
\begin{enumerate}[(i)]
 \item $\ga(\M)=0$. Consequently, by Proposition~\ref{prop.roots.gevrey.multiply.Mat.ind.order} and Corollary~\ref{coro:snq.ga1.gabeta} $\ga(\widehat{\M})=1$, so $\M$ does not satisfy \hyperlink{snq}{(snq)}  and $\widehat{\m}$ does not satisfy \hyperlink{gar}{$(\gamma_1)$}.
  \item $\o(\M)=\oo$. Hence $\ro(\M)=\a(\M)=\oo$ and  $\o(\widehat{\M})=\ro(\widehat{\M})=\a(\widehat{\M})=\oo$.
 \item $\b(\o_\M)=\mu(\o_\M)=\ro(\o_\M)=\a(\o_\M)=0$, then $\ga(\o_\M)=\oo$ and $\o_\M$  has \hyperlink{omsnq}{$(\omega_{\on{snq}})$} (see Corollary~\ref{coro:omsnq}),
 \item $\b(\o_{\widehat{\M}})=\mu(\o_{\widehat{\M}})=\ro(\o_{\widehat{\M}})=\a(\o_{\widehat{\M}})=0$, then $\ga(\o_{\widehat{\M}})=\oo$ and $\o_{\widehat{\M}}$  has also \hyperlink{omsnq}{$(\omega_{\on{snq}})$}. 
\end{enumerate}\para

{\itshape Note:} For any weight sequence $\L\approx\M$, due to the stability of the indices under equivalence, see Corollary~\ref{coro:stability.gamma.weak.equiv} and Remark~\ref{rema:stabilityOmM}, conditions (i) and (ii) also hold for $\L$. Moreover, since $\o_\M$ has  \hyperlink{om1}{$(\omega_{\on{1}})$} we deduce that $\o_\M\hyperlink{sim}{\sim}\o_\L$ and also conditions (iii) and (iv) hold true for $\o_\L$.

\subsection{Construction of \texorpdfstring{$\M$}{M}}
We give now the explicit construction of $\M$, inspired by the Representation Theorem~\ref{th:rep.ORVS}, and prove all the desired properties. For all $p\ge1$, we put
\begin{equation*}
m_p:=\exp\left(\sum_{k=1}^{p}\delta_k\right),\qquad m_0:=1,
\end{equation*}
for some sequence $(\delta_k)^\oo_{k=1}$ of nonnegative real numbers. By definition such sequence of quotients is nondecreasing, so $\M$ is \hyperlink{lc}{(lc)}. \para

The sequence $(\delta_k)_k$ is introduced as follows: First consider two sequences of natural numbers $(a_j)_{j\ge 1}$ and $(b_j)_{j\ge 1}$, which are defined recursively by
$a_1:=1$ and for every $j\in\N$, $b_j:=2a_j$ and $a_{j+1}:=b_j^2$. Then for all $j\in\N$ we get
\begin{equation*}
a_j=2^{2(2^{j-1}-1)},\qquad b_j=2^{2^j-1}.
\end{equation*}
The sequence $(\delta_k)_{k\ge 1}$ is defined now by
\begin{equation*}
\delta_k:=c_j=2^{2^{j+1}} \quad \text{if}\;a_j+1\le k\le b_j;\qquad\delta_p:=0\quad \text{if}\;b_j+1\le k\le a_{j+1},\quad j\ge 1.
\end{equation*}
{\itshape First immediate consequence:} We observe that $m_{a_j}\geq \exp(\de_{a_j})= \exp (c_j)$ then $\lim_{p\to\oo} m_p=\oo$ and $\M$ is a weight sequence. \para

{\itshape Second immediate consequence:} $\ga(\M)=0$, i.e., (i) holds.  For all $k\in\N$, $kb_j<a_{j+1}$ for all $j\ge 1$ large enough (depending on given $k$). 
For all $k\in\N$ and all $j\ge 1$ large enough $ m_{kb_j}/m_{b_p}=1$ and, by Theorem~\ref{th:main.beta.sequences}, we conclude that $\ga(\M)=0$.\para

For convenience we put $L_p:=\log(m_p)/p$,  for all $p\ge 1$, and we will show that 
\begin{equation}\label{eq:Lpinfinity}
 \lim_{p\to\oo} L_p =\lim_{p\to\oo} \frac{\log(m_p)}{p} =\lim_{p\to\oo}  \frac{1}{p} \sum_{j=1}^p \de_j =\oo. 
\end{equation}
For all $j\geq 1$ we get
$$L_{b_j}=\frac{1}{b_j}\sum_{i=1}^j(b_i-a_i)c_i=\frac{1}{b_j}\sum_{i=1}^ja_ic_i,\qquad L_{a_{j+1}}=\frac{1}{a_{j+1}}\sum_{i=1}^{j}(b_i-a_i)c_i=\frac{1}{a_{j+1}}\sum_{i=1}^{j}a_ic_i.$$
By definition $p\mapsto L_p$ is nonincreasing on each $[b_j,a_{j+1}]$, $j\ge 1$. With a direct computation, we can show that the choice of $c_j$ is sufficient for $p\mapsto L_p$ being nondecreasing on $[a_j,b_j]$, for all $j\ge 2$. An easy calculation leads to $2L_{b_j}\geq c_j \geq L_{a_j}\geq c_{j-2}$ for all $j\geq3$. Hence \eqref{eq:Lpinfinity} is valid.\para

{\itshape Third immediate consequence:} $\o(\M)=\oo$ because $\lim_{p\to\oo} \log(m_p)/\log p =\oo$ by~\eqref{eq:Lpinfinity}. By Remark~\ref{rema.Ch.Th.ORVSEq} and Proposition~\ref{prop.roots.gevrey.multiply.Mat.ind.order}, condition (ii) is valid and we conclude that neither $\M$ nor $\widehat{\M}$ have \hyperlink{mg}{(mg)}.\para

\subsection{Proving \texorpdfstring{$\gamma(\o_{\widehat{\M}})=\oo$}{Ga(OmM)}}

The most arduous part of the example is the proof of $\ga(\o_{\widehat{\M}})=\oo$. The goal is to show that  $\o_{\widehat{\M}}$ has $(P_{\omega_{\widehat{\M}},1/s})$ for all $s>0$ small enough. 
Using \eqref{equarelaMdetnuder} and that $\gamma(\omega_{\widehat{\M}})>1/s$ if and only if $\gamma((\omega_{\widehat{\M}})_{1/s})>1$, where $(\o_{{\widehat{\M}}})_{1/s}(t)=\o_{{\widehat{\M}}}(t^{1/s})$,
 we can show that  for $\omega_{\widehat{\M}}$ to satisfy $(P_{\omega_{\widehat{\M}},1/s})$, it suffices to prove that there exists $C\geq1$ such that for all $p\in\N$
\begin{equation}\label{eq:POms.sufficient}
1+\frac{(\widehat{m}_{p+1})^s}{p}\sum_{j=p}^{\infty}\frac{1}{(\widehat{m}_{j+1})^{s}}\le (C-s)\log\left(\frac{\widehat{m}_p}{(\widehat{M}_p)^{1/p}}\right)+\frac{C}{p}.
\end{equation}
With a careful computation one can show that \eqref{eq:POms.sufficient} holds for all $s\in(0,1)$. Then $\ga(\o_{\widehat{\M}})=\oo$.\para

{\itshape Final consequence:} $\ga(\o_{\M})=\oo$ by Corollary~\ref{coro:OmhatMOmMgamma}. Hence (iii) and (iv) hold by Remark~\ref{rema:ordersROMU}, Lemma~\ref{lemma.alpha.gamma} and Corollary~\ref{coro:omsnq}.\para

\textbf{Acknowledgements}: The first two authors are partially supported by the Spanish Ministry of Economy, Industry and Competitiveness under the project MTM2016-77642-C2-1-P. The first author is partially supported by the University of Valladolid through a Predoctoral Fellowship (2013 call) co-sponsored by the Banco de Santander. The third author is supported by FWF-Project J~3948-N35, as a part of which he is an external researcher at the Universidad de Valladolid (Spain) for the period October 2016 - September 2018.\para

\bibliographystyle{plain}
\bibliography{Bibliography}
\vskip1cm

\textbf{Affiliation}:\\
J.~Jim\'{e}nez-Garrido, J.~Sanz:\\
Departamento de \'Algebra, An\'alisis Matem\'atico, Geometr{\'\i}a y Topolog{\'\i}a, Universidad de Valladolid\\
Facultad de Ciencias, Paseo de Bel\'en 7, 47011 Valladolid, Spain.\\
Instituto de Investigaci\'on en Matem\'aticas IMUVA\\
E-mails: jjjimenez@am.uva.es (J.~Jim\'{e}nez-Garrido), jsanzg@am.uva.es (J. Sanz).\\
\vspace{5pt}\\
G.~Schindl:\\
Departamento de \'Algebra, An\'alisis Matem\'atico, Geometr{\'\i}a y Topolog{\'\i}a, Universidad de Valladolid\\
Facultad de Ciencias, Paseo de Bel\'en 7, 47011 Valladolid, Spain.\\
E-mail: gerhard.schindl@univie.ac.at.

\end{document}